\documentclass[12pt,reqno]{amsart}
\usepackage{amsmath,amssymb,amsthm}
\usepackage{mathrsfs}
\usepackage{xcolor} 
\usepackage[utf8]{inputenc}
\usepackage[T1]{fontenc}
\usepackage{appendix}
\usepackage{enumitem}
\usepackage[protrusion=true,expansion=false]{microtype}
\usepackage{tikz-cd}
\usepackage[margin=1in]{geometry}
\usepackage{mathtools}
 \usepackage[colorlinks=true, linkcolor=blue, citecolor=blue]{hyperref}    
\usepackage{orcidlink}
  \newcommand{\orcid}[1]{\,{\large\orcidlink{#1}}}
\usepackage[capitalize,nameinlink]{cleveref}
\newtheorem{theorem}{Theorem}[section]
\newtheorem*{acknow}{\textup{Acknowledgement}}

\newtheorem{lemma}[theorem]{Lemma}

\theoremstyle{definition}

\newtheorem{remark}[theorem]{Remark}

\newtheorem{corollary}[theorem]{Corollary}

\newcommand\Q{\mathbb{Q}}
\newcommand\R{\mathbb{R}}
\newcommand\Z{\mathbb{Z}}
\newcommand\C{\mathbb{C}}

\newcommand{\HP}{\mathbb{HP}}
\newcommand{\AffiliationsBlock}{
  \vspace{0.8em}
  \begingroup
  \centering
  \footnotesize\itshape
  $^{a}$ Department of Mathematics, Indian Institute of Technology Madras, \\
  Chennai-600036, Tamil Nadu, India.\par
  \endgroup
  \vspace{-1.3em}
}\makeatletter
\patchcmd{\@maketitle}{
  \ifx\@empty\@dedicatory
  \else
}{
  \AffiliationsBlock
  \vspace{1.5em}
  \ifx\@empty\@dedicatory
  \else
}{}{\message{Patch failed!}}
\makeatother

\numberwithin{equation}{section}

\usepackage{orcidlink}

\begin{document}

\title[Complex structures on bundles over $M_{\ell,m}$]{Complex and stable complex
structures on real vector bundles over connected sums of quaternionic projective planes}
\author[S. Mandal]{{Souvik Mandal$^{\,a,*}$}
 \orcid{0009-0004-0429-3063}
}
\thanks{\hspace{-1.1em}*
  \makebox[8.5em][l]{\textit{E-mail addresses:}}
  \begin{minipage}[t]{0.75\textwidth}
    \texttt{ma22d014@smail.iitm.ac.in}, \texttt{ssouvik.xyz@gmail.com}.
  \end{minipage}}

\begin{abstract}
Let $M_{\ell,m}=\ell\,\mathbb{HP}^{2}\# m\,\overline{\mathbb{HP}^{2}}$ be the connected sum of
$\ell$ copies of $\mathbb{HP}^{2}$ with $m$ copies of $\mathbb{HP}^{2}$ endowed with the opposite orientation. We characterise, in terms of Pontryagin classes, the real vector bundles over $M_{\ell,m}$ admitting a stable complex structure, and deduce that $M_{\ell,m}$ is stably almost complex if and only if $\ell-m$ is even. We then determine when an oriented real vector bundle of even rank over $M_{\ell,m}$ admits a complex structure. Consequently, $M_{\ell,m}$ admits an almost complex structure if and only if $\ell=2m+1$ and $m$ is odd. This corrects an assertion made by Sato and Suzuki in 1974, according to which
$M_{\ell,m}$ is never almost complex. We also identify the unjustified step in
their argument.
\end{abstract}
\maketitle
\vspace{-0.5em}
\begin{center}
\begin{minipage}{0.845\textwidth}
    \footnotesize
    
    \begin{list}{}{
        \leftmargin=5.5em 
        \labelwidth=5.5em
        \labelsep=0pt \parsep=0pt \topsep=0pt \itemsep=0pt
    }
        \item[\textit{Keywords:}\hfill] Complex structure, Stable complex structure, Almost complex structure, Stable almost complex structure, Quaternionic projective plane, Connected sum.
    \end{list}

    \vspace{4pt}

    \begin{list}{}{
        \leftmargin=18.5em 
        \labelwidth=18.5em
        \labelsep=0pt \parsep=0pt \topsep=0pt \itemsep=0pt
    }
        \item[2020\hspace{1mm}\textit{Mathematics Subject Classification:}\hfill] {Primary 32Q60, 57R22;\\ Secondary 19L64, 55R50, 55S35, 57R15, 57R20.}
    \end{list}
\end{minipage}
\end{center}
\vspace{1em}

\section{Introduction}\label{sec:intro}

The existence of complex and stable complex structures on real vector bundles is
a fundamental problem in topology, investigated extensively by Massey~\cite{Massey},
Thomas~\cite{Thomas}, and Albanese and Milivojevi\'c~\cite{albanesemassey}, among
others, in terms of obstruction theory, characteristic classes and secondary
cohomology operations.

For the quaternionic projective spaces $\HP^{n}$, the non-existence of almost
complex structures was established in the decade following $1949$. The case $n=1$,
in which $\HP^{1}\cong\mathbb{S}^{4}$, was settled by Ehresmann~\cite{Ehresmann}
and by Hopf~\cite{Hopf}. For $n\geq 4$, Hirzebruch~\cite{Hirzebruch1953} ruled
out the existence of such structures using characteristic classes. The remaining
cases, $n=2$ and $n=3$, were subsequently resolved by Milnor, a result
announced during Hirzebruch's address at the 1958 International Congress of
Mathematicians~\cite{Hirzebruch1958}. Milnor's proof has remained unpublished,
and a self-contained proof for all $n\geq 2$ was later provided by
Massey~\cite{Massey1962}, who computed the ring $\widetilde{K}(\HP^{n})$ and the
Chern character of the tangent bundle of $\HP^{n}$. Gauduchon, Moroianu and
Semmelmann~\cite[Theorem~1.1]{GMS} have since shown that $\HP^{n}$ is not even
stably almost complex for $n\geq 2$; the case $n=1$ is exceptional, the sphere
$\mathbb{S}^{4}$ being stably parallelisable and hence stably almost complex.

The existence of almost complex structures on connected sums has received
considerable attention, in particular for connected sums of complex projective
spaces, studied by Goertsches and Konstantis~\cite{acsoscocps} and by
Yang~\cite{yang2019almost}; further results in this direction are due to
Albanese and Milivojevi\'c~\cite{albanese2019connected}. For connected sums of quaternionic projective spaces the existence problem has been
regarded as settled since 1974, when Sato and Suzuki \cite[Theorem~B]{SatoSuzuki}
asserted that no connected sum $\ell\,\HP^{n}\mathbin{\#}m\,\overline{\HP^{n}}$
with $n\leq 10$ admits an almost complex structure, except possibly when $n=3$
and $\ell=3m+1$. However, their assertion fails for $n=2$. In fact
$\ell\,\HP^{2}\#\,m\,\overline{\HP^{2}}$ admits an almost complex structure
compatible with its orientation precisely when $\ell=2m+1$ and $m$ is
odd (Corollary~\ref{thm:acsclassification}). Remark~\ref{rem:SatoSuzuki} identifies
the unjustified step in their argument, and verifies that the criterion of
Heaps~\cite[Theorem~1]{Heaps} to which Sato and
Suzuki appeal for the case $n=2$ returns for $M_{\ell,m}$ precisely the two
conditions of Corollary~\ref{thm:acsclassification}.

A \emph{complex structure} on an oriented real vector bundle $\xi$ of rank $2n$ over a topological space $X$ is given by a bundle endomorphism $J$ of $\xi$ with $J^{2}=-\mathrm{id}$ inducing the orientation of $\xi$. Equivalently, $\xi$ admits a complex structure if there exists a complex vector bundle of rank $n$ over $X$ whose underlying real vector bundle is isomorphic to $\xi$ as an oriented bundle. More generally, a \emph{stable complex structure} on a real vector bundle $\xi$ over $X$ is defined by a complex vector bundle over $X$ whose underlying real vector bundle is stably isomorphic to $\xi$. For $X$ a finite connected CW complex, an oriented real vector bundle $\xi$ of rank $2n$ over $X$ admits a complex structure precisely when the homotopy class of its classifying map lies in the image of the map $[X,BU(n)]\to[X,BSO(2n)]$ induced by the canonical inclusion of the unitary group $U(n)$ in the special orthogonal group $SO(2n)$, where $BG$ denotes the classifying space of a topological group $G$. Similarly, writing $\widetilde{K}(X)$ and $\widetilde{KO}(X)$ for the reduced complex and real $K$-groups, $\xi$ admits a stable complex structure precisely when its stable class $\widetilde{\xi}\in\widetilde{KO}(X)$ lies in the image of the realification map $\widetilde{r}:\widetilde{K}(X)\to \widetilde{KO}(X)$, induced by passing from a complex vector bundle to its underlying real vector bundle. An oriented smooth manifold $M$ is said to admit an \emph{almost complex structure}
if its tangent bundle $TM$ admits a complex structure, and a \emph{stable almost
complex structure} if $TM$ admits a stable complex structure.

In this article we study the existence of complex and stable complex structures on bundles over the family of closed oriented manifolds
\[
M_{\ell,m}=\ell\,\HP^{2}\#\,m\,\overline{\HP^{2}},
\qquad \ell,m\ge 0,\ \ell+m\ge 1,
\]
where $\overline{\HP^{2}}$ denotes $\HP^{2}$ equipped with the opposite orientation, and in each case we deduce the corresponding statement for the manifolds themselves.

We first fix the notation in which our results are stated; it is established in
Lemma~\ref{lem:cohomology}. For a closed oriented manifold $N$ we write $[N]$ for its fundamental class. Put $k=\ell+m$. The cohomology
$H^{*}(M_{\ell,m};\Z)$ is torsion-free and concentrated in degrees $0$, $4$
and $8$; there is a basis $\{x_{1},\dots,x_{k}\}$ of
$H^{4}(M_{\ell,m};\Z)$ and a generator $u$ of $H^{8}(M_{\ell,m};\Z)$ with
$\langle u,[M_{\ell,m}]\rangle=1$ such that
\[
x_{i}x_{j}=0\ (i\neq j),
\qquad
x_{i}^{2}=\epsilon_{i}u,
\qquad
\epsilon_{i}=
\begin{cases}
+1, & 1\leq i\leq \ell,\\
-1, & \ell<i\leq k.
\end{cases}
\]
In particular $H^{2}(M_{\ell,m};\Z)=H^{6}(M_{\ell,m};\Z)=0$, and likewise with
$\Z_{2}$ coefficients.

The basis is normalised, once and for all, by requiring in addition that
\[
p_{1}(M_{\ell,m})=2\,(x_{1}+x_{2}+\cdots+x_{k});
\]
such a basis exists (Theorem~\ref{prop:p1ConnectedSum}), and the reason for imposing
this normalisation is explained after Theorem~\ref{maintheorem}.

Since $H^{2}(M_{\ell,m};\Z_{2})=0$, every real vector bundle $\xi$ over $M_{\ell,m}$
has $w_{2}(\xi)=0$, so that the mod~$2$ reduction of $p_{1}(\xi)$, which equals
$w_{2}(\xi)^{2}$, vanishes, and there are uniquely determined integers
$a_{1},\dots,a_{k}$ and $b$ with
\[
p_{1}(\xi)=2\sum_{i=1}^{k}a_{i}x_{i},
\qquad
p_{2}(\xi)=b\,u .
\]
For an oriented real vector bundle $\xi$ we write $e(\xi)$ for its Euler class; if
$\xi$ has rank $8$ we write in addition $e(\xi)=e\,u$ with $e\in\Z$.

Our first result decides the stable problem in these terms; it is proved in
Section~\ref{sec:stable} by a direct computation of $\widetilde{K}(M_{\ell,m})$, $\widetilde{KO}(M_{\ell,m})$ and of the realification map between them.

\begin{theorem}\label{maintheorem}
A real vector bundle $\xi$ over $M_{\ell,m}$ admits a stable complex
structure if and only if
\begin{equation}\label{maincong}
b\;\equiv\;\sum_{i=1}^{k}\epsilon_{i}a_{i}(2a_{i}-1)\pmod{4}.
\end{equation}
\end{theorem}

The normalisation of the basis $\{x_{1},\dots,x_{k}\}$ is not redundant, since
replacing a single $x_{i}$ by $-x_{i}$ preserves the relations $x_{i}x_{j}=0$ $(i\neq j)$ and $x_{i}^{2}=\epsilon_{i}u$ but
replaces $a_{i}$ by $-a_{i}$, and thereby changes the right-hand side
of the congruence~\eqref{maincong} by $2\epsilon_{i}a_{i}$, which is not divisible by $4$ when
$a_{i}$ is odd. The congruence \eqref{maincong} is nevertheless independent of the
choice of normalised basis,being equivalent by relations~\eqref{eq:ring}
to
\begin{equation*}
\bigl\langle\,2\,p_{1}(\xi)^{2}-4\,p_{2}(\xi)-p_{1}(\xi)\,p_{1}(M_{\ell,m})\,,\
[M_{\ell,m}]\bigr\rangle\equiv 0 \pmod{16},
\end{equation*}
which involves only $\xi$ and $M_{\ell,m}$. 

For the tangent bundle one has $a_{i}=1$ for all $i$ and $b=7(\ell-m)$, so
the congruence~\eqref{maincong} reduces to a congruence in $\ell$ and $m$ alone.

\begin{corollary}\label{cor:stablyAC}
The manifold $M_{\ell,m}$ admits a stable almost complex structure if and
only if $\ell-m$ is even.
\end{corollary}

By Theorem~\ref{thm:connectedsumtangent} the stable tangent bundle of $M\#N$
splits as $p_{M}^{*}(TM)\oplus p_{N}^{*}(TN)$, so that a connected sum of stably
almost complex manifolds is again stably almost complex.
Corollary~\ref{cor:stablyAC} does not follow from this, since $\HP^{2}$ is not
stably almost complex. It exhibits instead a family in which a stable almost
complex structure arises from the connected sum rather than from its summands,
the manifold $\HP^{2}\#\HP^{2}$ being stably almost complex although $\HP^{2}$
is not.

For a real vector bundle $\xi$ over $M_{\ell,m}$, the passage from stable complex
structures to complex structures is governed in ranks $4$ and $6$ by the half-spinor
bundles of $\xi$ (Lemma~\ref{lem:linesubbundle}), and in rank $8$ by the Chern classes of a stable complex structure
on $\xi$. Since $H^{2}(M_{\ell,m};\Z)$ and $H^{6}(M_{\ell,m};\Z)$ vanish, these Chern
classes are determined by the Pontryagin classes of $\xi$ (Lemma~\ref{lem:chern}), and
in particular independent of the stable complex structure chosen; the criteria below
therefore involve only the integers associated above to $\xi$, together with its
Euler class.
\begin{theorem}\label{thm:bundleacs}
Let $\xi$ be an oriented real vector bundle of rank $2n$ over $M_{\ell,m}$.
\begin{enumerate}[label=\textup{(\roman*)},leftmargin=2.6em]
\item If $2n=4$, then $\xi$ admits a complex structure if and only if
\begin{equation}\label{eq:rank4nec}
e(\xi)=-\sum_{i=1}^{k}a_{i}x_{i}.
\end{equation}
\item If $2n=6$, then $\xi$ admits a complex structure if and only if
\begin{equation}\label{eq:rank6}
b=\sum_{i=1}^{k}\epsilon_{i}a_{i}^{2}.
\end{equation}
\item If $2n=8$, then $\xi$ admits a complex structure if and only if congruence~\eqref{maincong} holds and
\begin{equation}\label{eq:euler}
b-\sum_{i=1}^{k}\epsilon_{i}a_{i}^{2}=2e.
\end{equation}
\item If $2n\geq 10$, then $\xi$ admits a complex structure if and only if congruence~\eqref{maincong} holds.
\end{enumerate}
\end{theorem} 
The theorem covers every even rank other than $2$. In rank $2$ every oriented real
vector bundle over $M_{\ell,m}$ is trivial, since $H^{2}(M_{\ell,m};\Z)=0$, and
therefore admits a complex structure.

Observe that, the second Pontryagin class of an oriented real vector bundle of
rank $4$ being the square of its Euler class, equation \eqref{eq:rank4nec} implies
$b=\sum_{i=1}^{k}\epsilon_{i}a_{i}^{2}$, which is equation~\eqref{eq:rank6}. As conditions on
arbitrary integer tuples, however, equation~\eqref{eq:rank6} does not
imply the congruence \eqref{maincong}, whereas a complex structure on a bundle is
in particular a stable complex structure, so that congruence~\eqref{maincong} holds for every
bundle admitting a complex structure. No contradiction arises, since not every tuple $(a_{1},\dots,a_{k};b)$ is realised by
a bundle of rank $6$. In Remark~\ref{rem:rank6realizability} we determine which tuples are so realised, and
deduce that for those arising from a bundle, both equations~\eqref{eq:rank4nec}
and \eqref{eq:rank6} do imply the congruence~\eqref{maincong}.

Applied to the tangent bundle, for which additionally $e=\chi(M_{\ell,m})=\ell+m+2$, Theorem~\ref{thm:bundleacs} determines in Section~\ref{sec:acs} exactly which of the manifolds $M_{\ell,m}$ admit an almost complex structure.

\begin{corollary}\label{thm:acsclassification}
The manifold $M_{\ell,m}$ admits an almost complex structure compatible with
its orientation if and only if $\ell=2m+1$ and $m$ is odd; equivalently, if
and only if $(\ell,m)=(4j+3,\,2j+1)$ for some integer $j\geq 0$.
\end{corollary}

In Remarks~\ref{rem:Yang2012stable} and~\ref{rem:Yang2012acs} we recover
Corollaries~\ref{cor:stablyAC} and~\ref{thm:acsclassification} from the criteria
of Yang~\cite{Yang2012}.

\subsection*{Organisation of the paper} Section~\ref{sec:prelim} records the cohomology ring of $M_{\ell,m}$, computes
the Pontryagin classes of its tangent bundle,  together with the
characteristic classes of the tautological quaternionic line bundle over
$\HP^{2}$.
Section~\ref{sec:stable} determines $\widetilde{K}(M_{\ell,m})$,
$\widetilde{KO}(M_{\ell,m})$ and the realification map between them, from which
Theorem~\ref{maintheorem} and Corollary~\ref{cor:stablyAC} follow.
Section~\ref{sec:bundleacs} passes from stable complex structures to complex
structures and proves Theorem~\ref{thm:bundleacs}, using the exceptional
isomorphisms $Spin(4)\cong SU(2)\times SU(2)$ and
$Spin(6)\cong SU(4)$ in ranks $4$ and $6$ respectively, and a criterion of
Thomas~\cite{Thomas} in rank $8$, together with the connectivity of
the maps of classifying spaces induced by stabilisation of the unitary and special
orthogonal groups in higher ranks. Section~\ref{sec:acs} applies
Theorem~\ref{thm:bundleacs} to the tangent bundle and proves
Corollary~\ref{thm:acsclassification}; Remark~\ref{rem:SatoSuzuki} examines the
argument of \cite[Theorem~B]{SatoSuzuki}. 
\section{Cohomology and Pontryagin classes of \texorpdfstring{$M_{\ell,m}$}{M\_\{l,m\}}}\label{sec:prelim}
In this section we normalise a generator of $H^{4}(\HP^{2};\Z)$ by means of the first
Pontryagin class of $\HP^{2}$, compute the characteristic classes of the tautological
quaternionic line bundle over $\HP^{2}$, record the cohomology ring of $M_{\ell,m}$,
and compute the Pontryagin classes of its tangent bundle.

Fix integers $\ell,m\geq 0$ with $k:=\ell+m\geq 1$ and set $M_{\ell,m}=\ell\,\HP^{2}\,\#\,m\,\overline{\HP^{2}}$. Recall that $H^{*}(\HP^{2};\Z)\cong\Z[v]/(v^{3})$ with $|v|=4$; in particular the
cohomology of $\HP^{2}$ is torsion-free and concentrated in degrees $0,4,8$.
Throughout, $\HP^{2}$ carries its standard orientation, for which
$\langle v^{2},[\HP^{2}]\rangle=1$ for either generator $v$ of $H^{4}(\HP^{2};\Z)$.

\begin{lemma}\label{firstp1}
There is a unique generator $\alpha\in H^{4}(\HP^{2};\Z)$ such that $p_{1}(\HP^{2})=2\alpha$; for this generator, $p_{2}(\HP^{2})=7\alpha^{2}$.
\end{lemma}

\begin{proof}
By \cite{quaternionic} there is a generator $w$ of $H^{4}(\HP^{n};\Z)$ for which the
total Pontryagin class satisfies $p(\HP^{n})=(1+w)^{2n+2}(1+4w)^{-1}$. Expanding
for $n=2$ and using $w^{3}=0$ yields $p(\HP^{2})=1+2w+7w^{2}$, so that
$p_{1}(\HP^{2})=2w$ and $p_{2}(\HP^{2})=7w^{2}$. The group $H^{4}(\HP^{2};\Z)\cong\Z$
has exactly two generators, $w$ and $-w$; the generator $\alpha=w$ is
therefore the desired one with $p_{1}(\HP^{2})=2\alpha$,
and $p_{2}(\HP^{2})=7\alpha^{2}$.
\end{proof}

Let $\gamma$ denote the tautological quaternionic line bundle over $\HP^{2}$
\cite[p.~243]{characteristic}, whose fibres we regard as right $\mathbb{H}$-modules;
right multiplication by $i\in\mathbb{H}$ endows it with the structure of a complex
vector bundle of rank $2$, so that $c_{1}(\gamma)=0$.

\begin{lemma}\label{lem:gammanormalisation}
The generator $\alpha$ of Lemma~\ref{firstp1} satisfies
\begin{equation}\label{eq:normalisation}
c_{2}(\gamma)=-\alpha .
\end{equation}
\end{lemma}

\begin{proof}
Write $c_{2}=c_{2}(\gamma)$ and let $y_{1},y_{2}$ be the Chern roots of $\gamma$. As
$c_{1}(\gamma)=0$ one has $y_{2}=-y_{1}$, so that, writing $y=y_{1}$, the Chern roots
of $\gamma$ are $\pm y$ and $c_{2}=-y^{2}$.

One has
$T\HP^{2}\cong\operatorname{Hom}_{\mathbb H}(\gamma,\gamma^{\perp})$ \cite[p.~248]{characteristic}, where $\gamma^{\perp}$ denotes the orthogonal complement of $\gamma$ with
respect to the standard quaternionic Hermitian metric on
$\varepsilon_{\mathbb{H}}^{3}$. As $\operatorname{Hom}_{\mathbb H}(\gamma,-)$ is
additive and $\gamma\oplus\gamma^{\perp}=\varepsilon_{\mathbb{H}}^{3}$,
\begin{equation}\label{eq:tangentHP2}
T\HP^{2}\oplus\operatorname{Hom}_{\mathbb H}(\gamma,\gamma)
\;\cong\;\operatorname{Hom}_{\mathbb H}\bigl(\gamma,\varepsilon_{\mathbb{H}}^{3}\bigr)
\;\cong\;3\,\gamma_{\R}
\end{equation}
as real vector bundles, the second isomorphism arising from the identification of
$\operatorname{Hom}_{\mathbb{H}}(\gamma,\varepsilon_{\mathbb{H}})$ with $\gamma_{\R}$
induced by the metric.

Every endomorphism of a one-dimensional right $\mathbb{H}$-module is left
multiplication by a scalar, so that $\operatorname{Hom}_{\mathbb{H}}(\gamma,\gamma)$
is a bundle of real algebras of rank $4$ with fibre $\mathbb{H}$. Taking the trace
of an endomorphism of the underlying real vector bundle defines a bundle map
$\operatorname{tr}:\operatorname{Hom}_{\mathbb H}(\gamma,\gamma)\to\varepsilon_{\mathbb{R}}$, and
$t\mapsto\tfrac{1}{4}\,t\cdot\mathrm{id}$ splits it, since
$\operatorname{tr}(\mathrm{id})=4$. Hence
\begin{equation}\label{eq:endsplit}
\operatorname{Hom}_{\mathbb H}(\gamma,\gamma)\;\cong\;\varepsilon_{\mathbb{R}}\oplus\psi ,
\end{equation}
the trivial summand being spanned by the identity, and $\psi=\ker\operatorname{tr}$
the rank-$3$ bundle of trace-free, equivalently purely imaginary, endomorphisms
of $\gamma$.

The complex structure of $\gamma$ is right multiplication by $i$, whereas the elements
of $\operatorname{Hom}_{\mathbb{H}}(\gamma,\gamma)$ act on the fibres of $\gamma$ by left multiplication; by
associativity the two commute, so that these elements are $\C$-linear.
Hence $\operatorname{Hom}_{\mathbb{H}}(\gamma,\gamma)$ is a real subbundle
of $\operatorname{End}_{\C}(\gamma)$, of half its real rank. Complexifying the inclusion of $\operatorname{Hom}_{\mathbb{H}}(\gamma,\gamma)$ in $\operatorname{End}_{\C}(\gamma)$
yields a homomorphism of bundles of complex algebras
whose source has fibre $\mathbb{H}\otimes_{\R}\C\cong M_{2}(\C)$
\cite[Ch.~I, Proposition~4.2]{LawsonMichelsohn}. This algebra being simple and the two bundles having the same complex rank, the
homomorphism is injective, hence bijective, on each fibre, and therefore an
isomorphism of bundles, so that by equation~\eqref{eq:endsplit},
\[
\varepsilon_{\mathbb{C}}\oplus(\psi\otimes\C)
\;\cong\;\operatorname{Hom}_{\mathbb H}(\gamma,\gamma)\otimes\C
\;\cong\;\operatorname{End}_{\C}(\gamma)\;\cong\;\gamma\otimes\gamma^{*},
\]
where $\gamma^{*}$ denotes the dual of $\gamma$.

 The Chern roots of $\gamma\otimes\gamma^{*}$ are the four sums of $\pm y$ with
$\mp y$, namely $2y,\,0,\,0,\,-2y$, whence
$c(\gamma\otimes\gamma^{*})=(1+2y)(1-2y)=1-4y^{2}$. This yields $c_{2}(\psi\otimes\C)=-4y^{2}=4c_{2}$ and
therefore
\[
p_{1}(\psi)=-c_{2}(\psi\otimes\C)=-4c_{2}.
\]
 
For the underlying real vector bundle of $\gamma$ one has
$p_{1}(\gamma_{\R})=c_{1}(\gamma)^{2}-2c_{2}(\gamma)=-2c_{2}$. Since $H^{*}(\HP^{2};\Z)$ is torsion-free, $p_{1}$ is additive on
equations~\eqref{eq:tangentHP2} and \eqref{eq:endsplit}, whence
\[
p_{1}(\HP^{2})=3\,p_{1}(\gamma_{\R})-p_{1}(\psi)=-6c_{2}+4c_{2}=-2\,c_{2}.
\]
Comparison with Lemma~\ref{firstp1}, according to which $p_{1}(\HP^{2})=2\alpha$, yields
$c_{2}(\gamma)=-\alpha$.
\end{proof}

Consequently the underlying real vector bundle $\gamma_{\R}$, of rank $4$, oriented
by the complex structure of $\gamma$, satisfies
\begin{equation}\label{eq:gammaclasses}
p_{1}(\gamma_{\R})=c_{1}(\gamma)^{2}-2c_{2}(\gamma)=2\alpha,
\qquad
p_{2}(\gamma_{\R})=c_{2}(\gamma)^{2}=\alpha^{2},
\qquad
e(\gamma_{\R})=c_{2}(\gamma)=-\alpha .
\end{equation}

We now fix notation on $M_{\ell,m}$. Let $u\in H^{8}(M_{\ell,m};\Z)\cong\Z$ be the
generator determined by the chosen orientation, i.e.,
$\langle u,[M_{\ell,m}]\rangle=1$, and set
\[
\epsilon_{i}=
\begin{cases}
+1,&1\le i\le \ell,\\
-1,&\ell<i\le k.
\end{cases}
\]
For $1\le i\le k$ let $q_{i}:M_{\ell,m}\to\HP^{2}$ denote the
standard collapse map onto the $i$th summand, obtained by collapsing the
complement of that summand to a point; it has degree $\epsilon_{i}$. Put
\[
x_{i}:=q_{i}^{*}(\alpha)\in H^{4}(M_{\ell,m};\Z),\qquad 1\le i\le k,
\]
with $\alpha$ as in Lemma~\ref{firstp1}.

Write $\nu:\mathbb{S}^{7}\to\mathbb{S}^{4}$ for the Hopf map,  
and set $\nu_{i}=\epsilon_{i}\nu$ for $1\leq i\leq k$. Up to
homotopy equivalence,
\[
M_{\ell,m}\;=\;\Bigl(\bigvee\nolimits_{k}\mathbb{S}^{4}\Bigr)\cup_{\varphi}\mathbb{D}^{8},
\]
where the attaching map $\varphi:\mathbb{S}^{7}\to\bigvee_{k}\mathbb{S}^{4}$ is the
composite of the pinch map $\mathbb{S}^{7}\to\bigvee_{k}\mathbb{S}^{7}$
with $\nu_{1}\vee\dots\vee\nu_{k}$. Thus $M_{\ell,m}$ fits into the
cofibre sequence
\begin{equation}\label{eq:cofibration}
\mathbb{S}^{7}\;\xrightarrow{\ \varphi\ }\;\bigvee_{k}\mathbb{S}^{4}
\;\xhookrightarrow{\ \iota\ }\;M_{\ell,m}\;\xrightarrow{\ q\ }\;\mathbb{S}^{8},
\end{equation}
where $q$ denotes the degree-one collapse map, and $q_{i}\circ\iota$ carries
the $i$th sphere by a homotopy equivalence onto $\HP^{1}\subset\HP^{2}$ and
collapses the remaining ones to a point.

\begin{lemma}\label{lem:cohomology}
The cohomology $H^{*}(M_{\ell,m};\Z)$ is torsion-free and concentrated in degrees $0,4,8$; the classes $x_{1},\dots,x_{k}$ form a basis of $H^{4}(M_{\ell,m};\Z)\cong\Z^{k}$, and
\begin{equation}\label{eq:ring}
x_{i}x_{j}=0 \ (i\neq j),\qquad
x_{i}^{2}=\epsilon_{i}\,u .
\end{equation}
\end{lemma}
\begin{proof}
Since the CW decomposition of $M_{\ell,m}$ described above consists
of one $0$-cell, $k$ $4$-cells and a single $8$-cell, the differentials of the
cellular cochain complex of $M_{\ell,m}$ vanish, so that $H^{*}(M_{\ell,m};\Z)$
is free and concentrated in degrees $0,4,8$, with
$H^{4}(M_{\ell,m};\Z)\cong\Z^{k}$ and $H^{8}(M_{\ell,m};\Z)\cong\Z$.

Since $\langle u,[M_{\ell,m}]\rangle=1$ and $H^{8}(M_{\ell,m};\Z)\cong\Z$, a class in
$H^{8}(M_{\ell,m};\Z)$ is determined by its Kronecker pairing with $[M_{\ell,m}]$,
so it is enough to evaluate the products $x_{r}x_{s}$ on the fundamental class.

Collapsing the separating spheres of the connected sum yields the canonical map
$c:M_{\ell,m}\to\bigvee\limits_{r=1}^{k}(\HP^{2})_{r}$ with $q_{r}=\pi_{r}\circ c$, where
$\pi_{r}$ denotes the projection onto the $r$th wedge summand. Products of
positive-dimensional classes coming from distinct summands of a wedge vanish, and
$c^{*}$ is a ring homomorphism; hence $x_{r}x_{s}=0$ for $r\neq s$.

For the squares, $x_{r}^{2}=q_{r}^{*}(\alpha^{2})$, and $q_{r}$ has degree
$\epsilon_{r}$, so that
\[
\bigl\langle x_{r}^{2},[M_{\ell,m}]\bigr\rangle
=\bigl\langle\alpha^{2},(q_{r})_{*}[M_{\ell,m}]\bigr\rangle
=\epsilon_{r}\bigl\langle\alpha^{2},[\HP^{2}]\bigr\rangle
=\epsilon_{r},
\]
that is, $x_{r}^{2}=\epsilon_{r}u$.

It remains to prove that $x_{1},\dots,x_{k}$ form a basis of $H^{4}(M_{\ell,m};\Z)$. Since $H^{*}(M_{\ell,m};\Z)$ is torsion-free, Poincar\'e duality implies that the pairing $(v,w)\mapsto\langle vw,[M_{\ell,m}]\rangle$ on $H^{4}(M_{\ell,m};\Z)$ is unimodular. By the relations just established, the Gram matrix of $x_{1},\dots,x_{k}$ with respect to this pairing is diagonal with entries $\epsilon_{1},\dots,\epsilon_{k}$, of determinant $(-1)^{m}$; in particular the classes $x_{1},\dots,x_{k}$ are linearly independent, and so span a subgroup of some finite index $d$ in $H^{4}(M_{\ell,m};\Z)\cong\Z^{k}$. The determinant of the pairing restricted to this subgroup equals $d^{2}$ times the determinant of the pairing on $H^{4}(M_{\ell,m};\Z)$, whence $d^{2}=1$ and the subgroup is all of $H^{4}(M_{\ell,m};\Z)$. Hence $x_{1},\dots,x_{k}$ is a basis, and relations~\eqref{eq:ring} holds. This completes the proof.
\end{proof}
Note that since the cellular cochain complex of $M_{\ell,m}$ with
$\Z_{2}$ coefficients is likewise concentrated in degrees $0,4,8$, we have
$H^{j}(M_{\ell,m};\Z_{2})=0$ for $j\neq 0,4,8$.

We now compute the Pontryagin classes of $M_{\ell,m}$; the result shows in
particular that the basis $\{x_{1},\dots,x_{k}\}$ of Lemma~\ref{lem:cohomology}
satisfies the normalisation  $p_{1}(M_{\ell,m})=2(x_{1}+\cdots+x_{k})$ imposed in
Section~\ref{sec:intro}.
\begin{theorem}\label{prop:p1ConnectedSum}
With $x_{i}=q_{i}^{*}(\alpha)$ as above,
\[
p_{1}(M_{\ell,m}) \;=\; 2\,(x_{1}+x_{2}+\cdots+x_{k}) \qquad \text{and}\qquad p_{2}(M_{\ell,m})=7(\ell-m)u.
\]
\end{theorem}
To prove this, we recall the description of the stable tangent bundle of the connected sum of two oriented manifolds, as detailed in \cite{acsoscocps}.

\begin{theorem}\cite{acsoscocps}\label{thm:connectedsumtangent}
Let $M$ and $N$ be two oriented smooth manifolds of the same dimension $d$. Let $p_{M}$ and $p_{N}$ denote the canonical collapsing maps $M\#N\rightarrow M$ and $M\#N\rightarrow N$ respectively. Then,
\[
T(M\#N)\oplus\varepsilon_{\mathbb{R}}^{d}\cong p_{M}^{*}(TM)\oplus p_{N}^{*}(TN)
\] as oriented real vector bundles.
\end{theorem}
\begin{proof}[Proof of Theorem~\ref{prop:p1ConnectedSum}]
Iterating Theorem~\ref{thm:connectedsumtangent} over the $k$ summands shows
that $TM_{\ell,m}$ is stably isomorphic to
$\bigoplus_{i=1}^{k}q_{i}^{*}(T\HP^{2})$. As the total Pontryagin class is
stable and, the cohomology of $M_{\ell,m}$ being torsion-free, multiplicative
under Whitney sum, Lemma~\ref{firstp1} yields
\[
p(M_{\ell,m})=\prod_{i=1}^{k}q_{i}^{*}\bigl(p(\HP^{2})\bigr)
=\prod_{i=1}^{k}\bigl(1+2x_{i}+7x_{i}^{2}\bigr).
\]
Every cross term in this product vanishes, since $x_{i}x_{j}=0$ for $i\neq j$
and $x_{i}u=0$; hence, using $x_{i}^{2}=\epsilon_{i}u$ and
$\sum_{i=1}^{k}\epsilon_{i}=\ell-m$,
\[
p(M_{\ell,m})=1+2\sum_{i=1}^{k}x_{i}+7\sum_{i=1}^{k}x_{i}^{2}
=1+2\sum_{i=1}^{k}x_{i}+7(\ell-m)u ,
\]
which is the assertion.
\end{proof}

\section{Existence of stable complex structures on bundles over \texorpdfstring{$M_{\ell,m}$}{M\_\{l,m\}}}\label{sec:stable}
We now derive a necessary and sufficient arithmetic condition for a bundle over
$M_{\ell,m}$ to admit a stable complex structure. The argument involves 
computations of the reduced complex and real $K$-groups $\widetilde{K}(M_{\ell,m})$
and $\widetilde{KO}(M_{\ell,m})$, together with the effect of realification on
Pontryagin classes.

Throughout this section we retain the notation of Sections~\ref{sec:intro}
and~\ref{sec:prelim}. Thus $\gamma$ denotes the tautological quaternionic line
bundle over $\HP^{2}$, normalised by equation~\eqref{eq:normalisation} and with
characteristic classes as in equation~\eqref{eq:gammaclasses}, and
$x_{i}=q_{i}^{*}(\alpha)$, with $q_{i}$ the collapse map of Section~\ref{sec:prelim}. We write $\operatorname{ch}$ for the
Chern character, with values in $H^{*}(M_{\ell,m};\Q)$, and $\operatorname{ch}_{j}$
for its component in $H^{2j}(M_{\ell,m};\Q)$.

We denote by $\xi$ a real vector bundle over $M_{\ell,m}$, and by
$a_{1},\dots,a_{k}$, $b$ the integers attached to it in Section~\ref{sec:intro},
so that
\begin{equation}\label{eq:invariants}
p_{1}(\xi)=2\sum_{i=1}^{k}a_{i}x_{i},
\qquad
p_{2}(\xi)=b\,u .
\end{equation}
Recall that they exist because $w_{2}(\xi)=0$; the mod~$2$ reduction of $p_{1}(\xi)$
equals $w_{2}(\xi)^{2}$ \cite[p.~181]{characteristic} and therefore vanishes, so that
$p_{1}(\xi)$ is divisible by $2$. They are unique because $\{x_{1},\dots,x_{k}\}$ is a
basis of the torsion-free group $H^{4}(M_{\ell,m};\Z)$.
\begin{lemma}\label{lem:Ktheory}
The group $\widetilde{K}(M_{\ell,m})$ is free abelian of rank $k+1$, with basis
$\{\eta_{1},\dots,\eta_{k},\theta\}$, where $\eta_{i}=q_{i}^{*}\bigl([\gamma]-2\bigr)$
and $\theta=q^{*}(\beta)$ for a suitable generator $\beta$ of
$\widetilde{K}(\mathbb{S}^{8})$. Moreover
\[
\operatorname{ch}(\eta_{i})=x_{i}+\tfrac{1}{12}\epsilon_{i}u,
\qquad
\operatorname{ch}(\theta)=u .
\]
\end{lemma}
\begin{proof}
Applying $\widetilde{K}(-)$ to the cofibre sequence~\eqref{eq:cofibration} yields an
exact sequence
\[
\widetilde{K}^{-1}\Bigl(\bigvee_{k}\mathbb{S}^{4}\Bigr)\to\widetilde{K}(\mathbb{S}^{8})
\xrightarrow{\,q^{*}\,}\widetilde{K}(M_{\ell,m})
\xrightarrow{\,\iota^{*}\,}\widetilde{K}\Bigl(\bigvee_{k}\mathbb{S}^{4}\Bigr)
\to\widetilde{K}^{1}(\mathbb{S}^{8}) .
\]
The first of the two outer groups is $\widetilde{K}(\bigvee_{k}\mathbb{S}^{5})$ by
the suspension isomorphism, and the second is $\widetilde{K}(\mathbb{S}^{9})$ by Bott
periodicity; both vanish, the reduced $K$-theory of an odd-dimensional sphere being
zero \cite{Bott}. The sequence therefore reduces to a
short exact sequence
\begin{equation}\label{eq:Kses}
0\to\widetilde{K}(\mathbb{S}^{8})
\xrightarrow{\,q^{*}\,}\widetilde{K}(M_{\ell,m})
\xrightarrow{\,\iota^{*}\,}\widetilde{K}\Bigl(\bigvee_{k}\mathbb{S}^{4}\Bigr)
=\bigoplus_{i=1}^{k}\widetilde{K}(\mathbb{S}^{4})\to 0 ;
\end{equation}
in particular $q^{*}$ is injective, and its image is the infinite cyclic group
generated by $\theta=q^{*}(\beta)$, where $\beta$ is a generator of
$\widetilde{K}(\mathbb{S}^{8})\cong\Z$ whose sign will be fixed below.

Since $c_{1}(\gamma)=0$ and $c_{2}(\gamma)=-\alpha$, the Chern roots of $\gamma$
are $\pm y$ with $y^{2}=\alpha$, whence
\begin{equation}\label{eq:chgamma}
\operatorname{ch}(\gamma)=e^{y}+e^{-y}=2+\alpha+\tfrac{1}{12}\alpha^{2} .
\end{equation}

Let $j:\HP^{1}\hookrightarrow\HP^{2}$ denote the inclusion, put
$\lambda=j^{*}\bigl([\gamma]-2\bigr)\in\widetilde{K}(\mathbb{S}^{4})$, and let
$\lambda_{i}$ be the corresponding element of the $i$th summand of $\bigoplus\limits_{i=1}^{k}\widetilde{K}(\mathbb{S}^{4})$. Applying $j^{*}$ to equation~\eqref{eq:chgamma} and using
$j^{*}(\alpha^{2})=\bigl(j^{*}(\alpha)\bigr)^{2}\in H^{8}(\mathbb{S}^{4};\Z)=0$ yields
$\operatorname{ch}(\lambda)=j^{*}(\alpha)$, which generates $H^{4}(\mathbb{S}^{4};\Z)$
because $j^{*}$ is an isomorphism on $H^{4}$, $\HP^{2}$ being obtained from
$\HP^{1}$ by attaching an $8$-cell. Since $\operatorname{ch}$ maps $\widetilde{K}(\mathbb{S}^{4})\cong\Z$ isomorphically
onto the image of $H^{4}(\mathbb{S}^{4};\Z)$ in $H^{4}(\mathbb{S}^{4};\Q)$
(see proof of Corollary~5.2 in~\cite{adamsvect}), $\lambda$ generates $\widetilde{K}(\mathbb{S}^{4})$, and
$\lambda_{1},\dots,\lambda_{k}$ is therefore a basis of
$\bigoplus_{i=1}^{k}\widetilde{K}(\mathbb{S}^{4})$.

Since $\eta_{i}=q_{i}^{*}\bigl([\gamma]-2\bigr)$, functoriality yields
$\iota^{*}(\eta_{i})=(q_{i}\circ\iota)^{*}\bigl([\gamma]-2\bigr)$. Now $q_{i}\circ\iota$ carries the $i$th sphere by a homotopy equivalence
onto $\HP^{1}$ and collapses the remaining ones to a point; its restriction to
the $i$th summand therefore pulls $[\gamma]-2$ back to $\pm\lambda$, while the other
summands contribute $0$. Thus $\iota^{*}(\eta_{i})$ generates the $i$th summand
of $\bigoplus\limits_{i=1}^{k}\widetilde{K}(\mathbb{S}^{4})$, and hence $\iota^{*}(\eta_{i})\mapsto\eta_{i}$ defines a section $s$
of $\iota^{*}$, so that exact sequence~\eqref{eq:Kses} splits;
\[
\widetilde{K}(M_{\ell,m})
=q^{*}\bigl(\widetilde{K}(\mathbb{S}^{8})\bigr)\oplus
s\Bigl(\widetilde{K}\Bigl(\bigvee_{k}\mathbb{S}^{4}\Bigr)\Bigr)
=\Z\{\theta\}\oplus\bigoplus_{i=1}^{k}\Z\{\eta_{i}\} ,
\]
which is free abelian of rank $k+1$ with the asserted basis.

It remains to compute Chern characters. Since $\operatorname{ch}$ is natural and
additive, $\operatorname{ch}(\eta_{i})=q_{i}^{*}\bigl(\operatorname{ch}(\gamma)-2\bigr)$;
substituting equation~\eqref{eq:chgamma}, applying the ring homomorphism $q_{i}^{*}$ and
using $x_{i}^{2}=\epsilon_{i}u$ yields
\[
\operatorname{ch}(\eta_{i})
=q_{i}^{*}\Bigl(\alpha+\tfrac{1}{12}\alpha^{2}\Bigr)
=x_{i}+\tfrac{1}{12}x_{i}^{2}
=x_{i}+\tfrac{1}{12}\epsilon_{i}u .
\]
Similarly $\operatorname{ch}(\theta)=q^{*}\operatorname{ch}(\beta)$. Here
$\operatorname{ch}(\beta)$ generates $H^{8}(\mathbb{S}^{8};\Z)$, since $\operatorname{ch}$ maps $\widetilde{K}(\mathbb{S}^{8})\cong\Z$ isomorphically
onto the image of $H^{8}(\mathbb{S}^{8};\Z)$ in $H^{8}(\mathbb{S}^{8};\Q)$ (see proof of Corollary~5.2 in~\cite{adamsvect}), and $q^{*}$ carries either generator of $H^{8}(\mathbb{S}^{8};\Z)$ to $\pm u$,
because $q$ has degree one. Thus $\operatorname{ch}(\theta)=\pm u$, and replacing
$\beta$ by $-\beta$ if necessary we may assume $\operatorname{ch}(\theta)=u$. This
fixes the generator $\beta$ and completes the proof.
\end{proof}
\begin{lemma}\label{lem:KOtheory}
The group $\widetilde{KO}(M_{\ell,m})$ is free abelian of rank $k+1$, with basis
$\{\rho_{1},\dots,\rho_{k},\sigma\}$, where
$\rho_{i}=q_{i}^{*}\bigl([\gamma_{\R}]-4\bigr)$ and $\sigma=q^{*}(\tau)$ for a
suitable generator $\tau$ of $\widetilde{KO}(\mathbb{S}^{8})$. For
$\zeta=\sum_{i=1}^{k}m_{i}\rho_{i}+n\sigma$ one has
\begin{equation}\label{eq:KOpontryagin}
p_{1}(\zeta)=2\sum_{i=1}^{k}m_{i}x_{i},
\qquad
p_{2}(\zeta)=\Bigl(\sum_{i=1}^{k}\epsilon_{i}m_{i}(2m_{i}-1)+6n\Bigr)u .
\end{equation}
Consequently, writing $p_{1}(\zeta)=2\sum_{i=1}^{k}a_{i}x_{i}$ and
$p_{2}(\zeta)=b\,u$ as in Section~\ref{sec:intro}, the
map $\zeta\mapsto(a_{1},\dots,a_{k};b)$ is injective on
$\widetilde{KO}(M_{\ell,m})$, and a tuple $(a_{1},\dots,a_{k};b)$ arises from a
real vector bundle over $M_{\ell,m}$ if and only if
\begin{equation}\label{eq:realizable}
b\equiv\sum_{i=1}^{k}\epsilon_{i}a_{i}(2a_{i}-1)\pmod{6}.
\end{equation}
\end{lemma}
\begin{proof}
Applying $\widetilde{KO}(-)$ to the cofibre sequence \eqref{eq:cofibration} yields an
exact sequence
\[
\widetilde{KO}^{-1}\Bigl(\bigvee_{k}\mathbb{S}^{4}\Bigr)\to\widetilde{KO}(\mathbb{S}^{8})
\xrightarrow{\,q^{*}\,}\widetilde{KO}(M_{\ell,m})
\xrightarrow{\,\iota^{*}\,}\widetilde{KO}\Bigl(\bigvee_{k}\mathbb{S}^{4}\Bigr)
\to\widetilde{KO}^{1}(\mathbb{S}^{8}) .
\]
By the suspension isomorphism
$\widetilde{KO}^{-1}(\bigvee_{k}\mathbb{S}^{4})\cong\bigoplus_{k}\widetilde{KO}(\mathbb{S}^{5})$,
which vanishes, while $\widetilde{KO}^{1}(\mathbb{S}^{8})\cong KO^{-7}(\mathrm{pt})=0$
by Bott periodicity \cite{Bott}. So the sequence
reduces to a short exact sequence
\begin{equation}\label{eq:KOses}
0\to\widetilde{KO}(\mathbb{S}^{8})
\xrightarrow{\,q^{*}\,}\widetilde{KO}(M_{\ell,m})
\xrightarrow{\,\iota^{*}\,}\widetilde{KO}\Bigl(\bigvee_{k}\mathbb{S}^{4}\Bigr)
=\bigoplus_{i=1}^{k}\widetilde{KO}(\mathbb{S}^{4})\to 0 ;
\end{equation}
in particular $q^{*}$ is injective, and its image is the infinite cyclic group
generated by $\sigma=q^{*}(\tau)$, where $\tau$ is a generator of
$\widetilde{KO}(\mathbb{S}^{8})\cong\Z$ whose sign will be fixed below.

We now identify the two outer groups appearing in the exact sequence~\eqref{eq:KOses}. For every
$r\geq 1$ the group $\widetilde{KO}(\mathbb{S}^{4r})$ is infinite cyclic \cite{Bott}, and
$p_{r}$ maps it isomorphically onto the subgroup
\[
\kappa_{r}(2r-1)!\,H^{4r}(\mathbb{S}^{4r};\Z),\qquad
\kappa_{r}=
\begin{cases}
2,&r\ \text{odd},\\
1,&r\ \text{even},
\end{cases}
\]
of $H^{4r}(\mathbb{S}^{4r};\Z)$. Indeed, a stable bundle $\xi_{0}$ over $\mathbb{S}^{4r}$ admits a trivialisation on
the $(4r-1)$-skeleton, and the obstruction $\mathfrak{o}(\xi_{0})\in
H^{4r}\bigl(\mathbb{S}^{4r};\pi_{4r-1}(SO)\bigr)\cong\Z$ to extending it over the top
cell is the clutching class of $\xi_{0}$, whence the obstruction map $\mathfrak{o}\colon\widetilde{KO}(\mathbb{S}^{4r})\to
H^{4r}(\mathbb{S}^{4r};\Z)$ is an isomorphism. The assertion follows since
$p_{r}(\xi_{0})=\pm\kappa_{r}(2r-1)!\,\mathfrak{o}(\xi_{0})$ \cite[Lemma~2]{milnor}.

For $r=1$ this image is $2H^{4}(\mathbb{S}^{4};\Z)$. Put
$\mu=j^{*}\bigl([\gamma_{\R}]-4\bigr)\in\widetilde{KO}(\mathbb{S}^{4})$, with $\mu_{i}$ the
corresponding element of the $i$th summand of $\bigoplus\limits_{i=1}^{k}\widetilde{KO}(\mathbb{S}^{4})$. Applying $j^{*}$
to equation~\eqref{eq:gammaclasses} yields $p_{1}(\mu)=2j^{*}(\alpha)$, a generator of
$2H^{4}(\mathbb{S}^{4};\Z)$; hence $\mu$ generates $\widetilde{KO}(\mathbb{S}^{4})$, and
$\mu_{1},\dots,\mu_{k}$ is a basis of $\bigoplus_{i=1}^{k}\widetilde{KO}(\mathbb{S}^{4})$.

For $r=2$ the image is $6H^{8}(\mathbb{S}^{8};\Z)$, so that $p_{2}(\tau)$ generates it. Since $q$ has
degree one, $q^{*}$ carries either generator of $H^{8}(\mathbb{S}^{8};\Z)$ to $\pm u$, so
that $p_{2}(\sigma)=q^{*}p_{2}(\tau)=\pm 6u$ by naturality; replacing $\tau$ by
$-\tau$ if necessary, we may assume
\begin{equation}\label{eq:psigma}
p_{1}(\sigma)=0,\qquad p_{2}(\sigma)=6u ,
\end{equation}
the first equality holding because $H^{4}(\mathbb{S}^{8};\Z)=0$. This fixes the generator
$\tau$.

Since $\rho_{i}=q_{i}^{*}\bigl([\gamma_{\R}]-4\bigr)$, functoriality yields
$\iota^{*}(\rho_{i})=(q_{i}\circ\iota)^{*}\bigl([\gamma_{\R}]-4\bigr)$. Now $q_{i}\circ\iota$ restricts to a homotopy equivalence from the $i$th sphere
onto $\HP^{1}$ and collapses the remaining ones to a point, so it pulls
$[\gamma_{\R}]-4$ back to $\pm\mu$ on the $i$th summand and to $0$ on the others.
Thus $\iota^{*}(\rho_{i})$ generates the $i$th summand of $\bigoplus\limits_{i=1}^{k}\widetilde{KO}(\mathbb{S}^{4})$, and
hence $\iota^{*}(\rho_{i})\mapsto\rho_{i}$ defines a section $s$ of $\iota^{*}$, so
that the exact sequence~\eqref{eq:KOses} splits;
\[
\widetilde{KO}(M_{\ell,m})=q^{*}\bigl(\widetilde{KO}(\mathbb{S}^{8})\bigr)\oplus
s\Bigl(\widetilde{KO}\Bigl(\bigvee_{k}\mathbb{S}^{4}\Bigr)\Bigr)
=\Z\{\sigma\}\oplus\bigoplus_{i=1}^{k}\Z\{\rho_{i}\} ,
\]
which is free abelian of rank $k+1$ with the asserted basis.

We turn to the Pontryagin classes. As $H^{*}(M_{\ell,m};\Z)$ is torsion-free, the
total Pontryagin class defines a homomorphism from $\widetilde{KO}(M_{\ell,m})$
to the multiplicative group $1+H^{4}(M_{\ell,m};\Z)+H^{8}(M_{\ell,m};\Z)$. Hence, for
$\zeta=\sum_{i=1}^{k}m_{i}\rho_{i}+n\sigma$,
\[
p(\zeta)=\prod_{i=1}^{k}p(\rho_{i})^{m_{i}}\cdot p(\sigma)^{n} .
\]
By equation~\eqref{eq:gammaclasses} we have $p(\gamma_{\R})=1+2\alpha+\alpha^{2}$;
applying $q_{i}^{*}$ and using $x_{i}^{2}=\epsilon_{i}u$ yields
$p(\rho_{i})=1+2x_{i}+\epsilon_{i}u$, while $p(\sigma)=1+6u$ by equation~\eqref{eq:psigma}. Put $A_{i}=2x_{i}+\epsilon_{i}u$. By Lemma~\ref{lem:cohomology},
$A_{i}^{2}=4x_{i}^{2}=4\epsilon_{i}u$ and $A_{i}^{3}=0$, so that the binomial series
for $(1+A_{i})^{m_{i}}$ terminates after two terms and, the binomial coefficient
$\binom{m_{i}}{2}=m_{i}(m_{i}-1)/2$ being defined for every integer $m_{i}$,
\[
(1+A_{i})^{m_{i}}
=1+m_{i}A_{i}+\binom{m_{i}}{2}A_{i}^{2}
=1+2m_{i}x_{i}+\epsilon_{i}m_{i}(2m_{i}-1)u
\qquad (m_{i}\in\Z),
\]
the last equality by $m_{i}+4\binom{m_{i}}{2}=m_{i}(2m_{i}-1)$. Since
$x_{i}x_{j}=0$ for $i\neq j$ and $x_{i}u=0$, all cross terms in the product
vanish, and conditions~\eqref{eq:KOpontryagin} follows.

Comparing conditions~\eqref{eq:KOpontryagin} with $p_{1}(\zeta)=2\sum_{i}a_{i}x_{i}$ and
$p_{2}(\zeta)=b\,u$ shows that $m_{i}=a_{i}$ and
$6n=b-\sum_{i}\epsilon_{i}a_{i}(2a_{i}-1)$, so that $\zeta$ is determined by the
tuple $(a_{1},\dots,a_{k};b)$; this is the asserted injectivity. If moreover the tuple arises from a real vector bundle $\xi'$ over $M_{\ell,m}$, then
it is the tuple attached to $\zeta=[\xi']-\operatorname{rk}\xi'$,
and equation~\eqref{eq:realizable} follows since $n\in\Z$.

Conversely, assume the condition~\eqref{eq:realizable} holds. Then
\[
n=\tfrac{1}{6}\Bigl(b-\sum_{i=1}^{k}\epsilon_{i}a_{i}(2a_{i}-1)\Bigr)
\]
is an integer, so that $\zeta=\sum_{i=1}^{k}a_{i}\rho_{i}+n\sigma$ is a
well-defined element of $\widetilde{KO}(M_{\ell,m})$; by equations~\eqref{eq:KOpontryagin}
it satisfies $p_{1}(\zeta)=2\sum_{i=1}^{k}a_{i}x_{i}$ and $p_{2}(\zeta)=b\,u$. As
$M_{\ell,m}$ is compact, $\zeta=[\xi']-\operatorname{rk}\xi'$ for some real vector
bundle $\xi'$ over $M_{\ell,m}$; since $p_{1}(\xi')=p_{1}(\zeta)$ and
$p_{2}(\xi')=p_{2}(\zeta)$, the tuple $(a_{1},\dots,a_{k};b)$ arises from $\xi'$.
\end{proof}
A real vector bundle over $M_{\ell,m}$ admits a stable complex structure exactly
when its stable class lies in the image of the realification map
$\widetilde{r}:\widetilde{K}(M_{\ell,m})\to\widetilde{KO}(M_{\ell,m})$; it
therefore remains to determine that image in terms of the tuple
$(a_{1},\dots,a_{k};b)$.

\begin{lemma}\label{lem:realification}
Let $\omega=\sum_{i=1}^{k}d_{i}\eta_{i}+n'\theta$ be an element of
$\widetilde{K}(M_{\ell,m})$, where $\{\eta_{1},\dots,\eta_{k},\theta\}$ is the
basis of $\widetilde{K}(M_{\ell,m})$ given by Lemma~\ref{lem:Ktheory}. Then the
tuple attached to $\widetilde{r}(\omega)$ in Lemma~\ref{lem:KOtheory} is given by
\begin{equation}\label{eq:image_of_r}
a_{i}=d_{i},
\qquad
b=\sum_{i=1}^{k}\epsilon_{i}d_{i}(2d_{i}-1)-12n' .
\end{equation}
\end{lemma}
\begin{proof}
Since $H^{2}(M_{\ell,m};\Z)=H^{6}(M_{\ell,m};\Z)=0$ we have
$c_{1}(\omega)=c_{3}(\omega)=0$, so that
\[
\operatorname{ch}_{2}(\omega)=-c_{2}(\omega),
\qquad
\operatorname{ch}_{4}(\omega)=\tfrac{1}{12}\bigl(c_{2}(\omega)^{2}-2c_{4}(\omega)\bigr),
\]
while Lemma~\ref{lem:Ktheory} yields
\[
\operatorname{ch}(\omega)=\sum_{i=1}^{k}d_{i}x_{i}
+\Bigl(\tfrac{1}{12}\sum_{i=1}^{k}\epsilon_{i}d_{i}+n'\Bigr)u .
\]
Comparing the two expressions for $\operatorname{ch}_{2}(\omega)$ gives
$c_{2}(\omega)=-\sum_{i=1}^{k}d_{i}x_{i}$, whence
$c_{2}(\omega)^{2}=\bigl(\sum_{i=1}^{k}\epsilon_{i}d_{i}^{2}\bigr)u$. Writing
$c_{4}(\omega)=\delta u$ with $\delta\in\Z$ and comparing the two expressions
for $\operatorname{ch}_{4}(\omega)$ then leads to
\[
\sum_{i=1}^{k}\epsilon_{i}d_{i}^{2}-2\delta
=\sum_{i=1}^{k}\epsilon_{i}d_{i}+12n' ,
\qquad\text{that is,}\qquad
\delta=\sum_{i=1}^{k}\epsilon_{i}\binom{d_{i}}{2}-6n' .
\]
One has the identity
$\sum_{r}(-1)^{r}p_{r}\bigl(\widetilde{r}(\omega)\bigr)=c(\omega)\,c(\overline{\omega})$~\cite[Ch.~15]{characteristic}, where $\overline{\omega}$ denotes the conjugate of $\omega$.
Since the odd Chern classes vanish, we have $c(\omega)=c(\overline{\omega})$. This yields
$p_{1}(\widetilde{r}(\omega))=-2c_{2}(\omega)$ and
$p_{2}(\widetilde{r}(\omega))=c_{2}(\omega)^{2}+2c_{4}(\omega)$. Hence
\[
p_{1}\bigl(\widetilde{r}(\omega)\bigr)=2\sum_{i=1}^{k}d_{i}x_{i},
\qquad
p_{2}\bigl(\widetilde{r}(\omega)\bigr)
=\Bigl(\sum_{i=1}^{k}\epsilon_{i}d_{i}^{2}+2\delta\Bigr)u
=\Bigl(\sum_{i=1}^{k}\epsilon_{i}d_{i}(2d_{i}-1)-12n'\Bigr)u .
\]
Comparing these expressions with equation~\eqref{eq:invariants} yields the values of
$a_{i}$ and $b$ asserted in equation~\eqref{eq:image_of_r}. This completes the proof.
\end{proof}
Lemmas~\ref{lem:Ktheory},~\ref{lem:KOtheory} and~\ref{lem:realification} now combine
to yield Theorem~\ref{maintheorem}.

\begin{proof}[\textbf{Proof of Theorem~\ref{maintheorem}}]
Since $\{\rho_{1},\dots,\rho_{k},\sigma\}$ is a basis of
$\widetilde{KO}(M_{\ell,m})$, the stable class
$\widetilde{\xi}=[\xi]-\operatorname{rk}\xi$ can be written uniquely as
$\widetilde{\xi}=\sum_{i=1}^{k}m_{i}\rho_{i}+n\sigma$ with $m_{i},n\in\Z$, and
Lemma~\ref{lem:KOtheory} yields $m_{i}=a_{i}$ as well as
\begin{equation}\label{eq:bn}
6n=b-\sum_{i=1}^{k}\epsilon_{i}a_{i}(2a_{i}-1) .
\end{equation}

By Lemma~\ref{lem:Ktheory}, every element of $\widetilde{K}(M_{\ell,m})$ has the
form $\omega=\sum_{i=1}^{k}d_{i}\eta_{i}+n'\theta$ with $d_{i},n'\in\Z$, so $\xi$
admits a stable complex structure if and only if
$\widetilde{\xi}=\widetilde{r}(\omega)$ for some such $\omega$. A class in
$\widetilde{KO}(M_{\ell,m})$ is determined by the tuple attached to it as in
Lemma~\ref{lem:KOtheory}; hence  the equality $\widetilde{\xi}=\widetilde{r}(\omega)$ holds if and only if $\widetilde{r}(\omega)$
and $\widetilde{\xi}$ have the same tuple, which by equation~\eqref{eq:image_of_r} amounts
to
\[
d_{i}=a_{i}\quad(1\le i\le k),
\qquad
12n'=\sum_{i=1}^{k}\epsilon_{i}a_{i}(2a_{i}-1)-b .
\]
In view of relation~\eqref{eq:bn} the second condition reads $6n=-12n'$, and such an
$n'\in\Z$ exists precisely when $n$ is even.

Finally, the congruence~\eqref{maincong} asserts that $4$ divides
$b-\sum_{i=1}^{k}\epsilon_{i}a_{i}(2a_{i}-1)$, which equals $6n$ by
relation~\eqref{eq:bn}; this happens precisely when $n$ is even. Hence, the congruence~\eqref{maincong} holds if and only if $\xi$ admits a stable complex structure.
\end{proof}

We now apply Theorem~\ref{maintheorem} to the tangent bundle of $M_{\ell,m}$, which yields Corollary~\ref{cor:stablyAC}.
\begin{proof}[\textbf{Proof of Corollary~\ref{cor:stablyAC}}]
For $\xi=TM_{\ell,m}$, Theorem~\ref{prop:p1ConnectedSum} yields $a_{i}=1$ for
every $i$ and $b=7(\ell-m)$. As $a_{i}(2a_{i}-1)=1$ and
$\sum_{i=1}^{k}\epsilon_{i}=\ell-m$, the congruence \eqref{maincong} reads
\[
7(\ell-m)\equiv\ell-m\pmod{4},
\]
that is, $6(\ell-m)\equiv 0\pmod{4}$, or equivalently
$3(\ell-m)\equiv 0\pmod{2}$. This holds precisely when
$\ell-m$ is even, and the assertion follows from Theorem~\ref{maintheorem}.
\end{proof}
\begin{remark}\label{rem:Yang2012stable}
Since $M_{\ell,m}$ is $3$-connected and of dimension $8$, the criterion of
Yang~\cite[Theorem~1, case~3]{Yang2012} for $(n-1)$-connected $2n$-manifolds
applies with $n=4$. By Lemma~\ref{lem:cohomology} the signature of $M_{\ell,m}$ is $\ell-m$, and
the criterion requires the integer
\[
\frac{B_{2}+B_{1}}{B_{2}B_{1}}\cdot\frac{\ell-m}{4}=9(\ell-m)
\]
to be even, where $B_{1}$ and $B_{2}$ are the Bernoulli numbers, with $B_{1}=1/6$
and $B_{2}=1/30$~\cite[Appendix~B]{characteristic}. This recovers Corollary~\ref{cor:stablyAC}.
\end{remark}
\section{Complex structures on real vector bundles over \texorpdfstring{$M_{\ell,m}$}{M\_\{l,m\}}}\label{sec:bundleacs}

Theorem~\ref{maintheorem} decides when a real vector bundle over $M_{\ell,m}$
admits a stable complex structure. In this section, we determine when such a
bundle admits a complex structure. Throughout, a complex structure on
an oriented real vector bundle is understood to induce the given orientation.

The two notions are related by the following criterion of Thomas.

\begin{theorem}[{\cite[Theorem~1.7]{Thomas}}]\label{thm:destabilize}
Let $X$ be a complex with $\dim X\leq 2n$ such that $H^{2n}(X;\Z)$ has no
$2$-torsion, and let $\xi$ be an oriented real vector bundle of rank $2n$ over $X$.
Then $\xi$ admits a complex structure if and only if $\xi$ admits a stable
complex structure $\omega$ with $c_{n}(\omega)=e(\xi)$.
\end{theorem}

Lemma~\ref{lem:chern} below shows that over $M_{\ell,m}$ the Chern classes of a
stable complex structure on $\xi$ are determined by the Pontryagin classes of
$\xi$; consequently the condition $c_{n}(\omega)=e(\xi)$ in
Theorem~\ref{thm:destabilize} does not depend on which stable complex structure
$\omega$ is taken.

\begin{lemma}\label{lem:chern}
Let $\xi$ be a real vector bundle over $M_{\ell,m}$, with associated integers
$a_{1},\dots,a_{k}$ and $b$ as in equation~\eqref{eq:invariants}, and suppose
that $\xi$ admits a stable complex structure $\omega$. Then
\[
c_{1}(\omega)=c_{3}(\omega)=0,
\qquad
c_{2}(\omega)=-\sum_{i=1}^{k}a_{i}x_{i},
\qquad
c_{4}(\omega)=\frac{1}{2}\Bigl(b-\sum_{i=1}^{k}\epsilon_{i}a_{i}^{2}\Bigr)u,
\]
and $c_{j}(\omega)=0$ for $j\geq 5$. In particular the Chern classes of $\omega$
do not depend on the choice of stable complex structure on $\xi$, and the
integer $b-\sum_{i=1}^{k}\epsilon_{i}a_{i}^{2}$ is even.
\end{lemma}
\begin{proof}
Since $H^{2}(M_{\ell,m};\Z)=H^{6}(M_{\ell,m};\Z)=0$ we have
$c_{1}(\omega)=c_{3}(\omega)=0$, and $c_{j}(\omega)=0$ for $j\geq 5$ because
$H^{2j}(M_{\ell,m};\Z)=0$ in those degrees. Hence
$c(\omega)=1+c_{2}(\omega)+c_{4}(\omega)$, and $c(\overline{\omega})=c(\omega)$
as the odd Chern classes vanish.

For the underlying real vector bundle $\omega_{\R}$ one has
$\sum_{r}(-1)^{r}p_{r}(\omega_{\R})=c(\omega)\,c(\overline{\omega})=c(\omega)^{2}$,
while $p_{r}(\omega_{\R})=p_{r}(\xi)$, the two bundles being stably isomorphic.
Since $H^{j}(M_{\ell,m};\Z)=0$ for $j>8$, this reads
\begin{equation}\label{eq:chernpontryagin}
1-p_{1}(\xi)+p_{2}(\xi)
=\bigl(1+c_{2}(\omega)+c_{4}(\omega)\bigr)^{2}
=1+2c_{2}(\omega)+\bigl(c_{2}(\omega)^{2}+2c_{4}(\omega)\bigr).
\end{equation}
In degree $4$, equation~\eqref{eq:chernpontryagin} yields
$2c_{2}(\omega)=-p_{1}(\xi)=-2\sum_{i=1}^{k}a_{i}x_{i}$. As
$H^{4}(M_{\ell,m};\Z)$ is torsion-free, this yields
$c_{2}(\omega)=-\sum_{i=1}^{k}a_{i}x_{i}$, and therefore
\[
c_{2}(\omega)^{2}=\sum_{i=1}^{k}a_{i}^{2}x_{i}^{2}
=\Bigl(\sum_{i=1}^{k}\epsilon_{i}a_{i}^{2}\Bigr)u
\]
by Lemma~\ref{lem:cohomology}. In degree $8$ equation~\eqref{eq:chernpontryagin} yields
\[
2c_{4}(\omega)=p_{2}(\xi)-c_{2}(\omega)^{2}
=\Bigl(b-\sum_{i=1}^{k}\epsilon_{i}a_{i}^{2}\Bigr)u ,
\]
which determines $c_{4}(\omega)$ since $H^{8}(M_{\ell,m};\Z)\cong\Z$ is
torsion-free; and as $c_{4}(\omega)$ is integral,
$b-\sum_{i=1}^{k}\epsilon_{i}a_{i}^{2}$ is even.
\end{proof}

Before proving Theorem~\ref{thm:bundleacs} we establish the following lemma, which
yields the converse in case~\textup{(i)}.

\begin{lemma}\label{lem:su2}
Every $SU(2)$-bundle over $M_{\ell,m}$ with vanishing second Chern class is
trivial.
\end{lemma}
\begin{proof}
Isomorphism classes of $SU(2)$-bundles over a CW-complex $X$ are in bijection
with $[X,BSU(2)]$, the trivial bundle corresponding to the homotopy class of
the constant map, and for $X=\mathbb{S}^{4}$ the second Chern class defines a
bijection $[\mathbb{S}^{4},BSU(2)]\cong H^{4}(\mathbb{S}^{4};\Z)$
\cite[Theorem~E.5, p.~221]{FreedUhlenbeck}. Since homotopy classes of maps
from a wedge are determined by their restrictions to the wedge summands, a map
$\bigvee\limits_{k}\mathbb{S}^{4}\to BSU(2)$ classifying an $SU(2)$-bundle with
vanishing second Chern class is null-homotopic.

Let $\zeta$ be an $SU(2)$-bundle over $M_{\ell,m}$ with $c_{2}(\zeta)=0$ and
let $f:M_{\ell,m}\to BSU(2)$ be its classifying map. The composite
$f\circ\iota$ classifies $\iota^{*}\zeta$, and
$c_{2}(\iota^{*}\zeta)=\iota^{*}c_{2}(\zeta)=0$, so that $f\circ\iota$ is
null-homotopic. Applying $[-,BSU(2)]$ to the cofibre sequence~\eqref{eq:cofibration} yields an exact sequence of pointed sets
\[
\Bigl[\bigvee_{k}\mathbb{S}^{5},BSU(2)\Bigr]
\xrightarrow{\ (\Sigma\varphi)^{*}\ }\bigl[\mathbb{S}^{8},BSU(2)\bigr]
\xrightarrow{\ q^{*}\ }\bigl[M_{\ell,m},BSU(2)\bigr]
\xrightarrow{\ \iota^{*}\ }\Bigl[\bigvee_{k}\mathbb{S}^{4},BSU(2)\Bigr],
\]
whose first two terms are groups and in which $(\Sigma\varphi)^{*}$ is a
homomorphism,
the map $\Sigma\varphi$ being a suspension. As $f\circ\iota$ is null-homotopic,
the homotopy class of $f$ lies in the image of $q^{*}$. Now
$(\Sigma\varphi)^{*}$ is surjective, as we prove below, so that by exactness
the image of $q^{*}$ consists of the base point alone. Hence $f$ is
null-homotopic, and consequently $\zeta$ is trivial.

By the suspension-loop adjunction, together
with the homotopy equivalence $\Omega BSU(2)\simeq SU(2)=\mathbb{S}^{3}$, the
homomorphism
$(\Sigma\varphi)^{*}:\bigl[\bigvee\limits_{k}\mathbb{S}^{5},BSU(2)\bigr]
\to\bigl[\mathbb{S}^{8},BSU(2)\bigr]$
is identified with the homomorphism
\[
\varphi^{*}:\bigoplus_{i=1}^{k}\pi_{4}(\mathbb{S}^{3})
\longrightarrow\pi_{7}(\mathbb{S}^{3}),
\]
and it therefore suffices to prove that $\varphi^{*}$ is surjective.

The map $\varphi$ being the composite of the pinch map
$\mathbb{S}^{7}\to\bigvee\limits_{k}\mathbb{S}^{7}$ with $\nu_{1}\vee\dots\vee\nu_{k}$,
where $\nu_{i}=\epsilon_{i}\nu$ (Section~\ref{sec:prelim}), the restriction of
$\varphi^{*}$ to the $i$th summand is $\nu_{i}^{*}$. 
Since $\pi_{4}(\mathbb{S}^{3})\cong\Z_{2}$ is generated by $\Sigma\eta$, where $\eta\colon\mathbb{S}^{3}\rightarrow\mathbb{S}^{2}$ denotes the Hopf map, the image
of $\varphi^{*}$ is generated by $\nu_{i}^{*}(\Sigma\eta)=\Sigma\eta\circ\nu_{i}=\Sigma\eta\circ\nu$, the sign $\epsilon_{i}$ being
immaterial in $\pi_{7}(\mathbb{S}^{3})\cong\Z_{2}$. By \cite[(5.9)]{Toda} one has
$(\Sigma\eta)\circ\nu=\nu'\circ\Sigma^{4}\eta$, where $\nu'$ generates the $2$-primary component of 
$\pi_{6}(\mathbb{S}^{3})\cong\Z_{12}$, and this
composite generates $\pi_{7}(\mathbb{S}^{3})\cong\Z_{2}$
\cite[Proposition~5.8]{Toda}. Hence $\varphi^{*}$ is surjective, and therefore so is $(\Sigma\varphi)^{*}$.
\end{proof}
\begin{remark}\label{rem:lowranks}
For $k=1$ an analogous statement for $SO(4)$-bundles is due to Sasao and
Tamura \cite{SasaoTamura}, who show that an $SO(4)$-bundle over $\HP^{2}$ is trivial
precisely when its Pontryagin and Euler classes vanish. For an $SU(2)$-bundle $\zeta$
over $M_{\ell,m}$ one has $p_{1}(\zeta_{\R})=-2c_{2}(\zeta)$
and $e(\zeta_{\R})=c_{2}(\zeta)$, so that the hypothesis of Lemma~\ref{lem:su2} is the
vanishing of the Pontryagin and Euler classes of $\zeta_{\R}$; Lemma~\ref{lem:su2}
strengthens the conclusion to triviality as an $SU(2)$-bundle, and applies to
every $M_{\ell,m}$.
\end{remark}

Since $H^{1}(M_{\ell,m};\Z_{2})=H^{2}(M_{\ell,m};\Z_{2})=0$, every oriented real
vector bundle over $M_{\ell,m}$ is spin and carries a unique spin
structure \cite[Ch.~II]{LawsonMichelsohn}. We recall the half-spinor bundles of an oriented real vector bundle carrying a
spin structure.

Let $n\geq 1$ and let $\Delta$ be an irreducible module over the complexified Clifford
algebra of $\R^{2n}$. The group $Spin(2n)$ lies in the Clifford algebra of $\R^{2n}$,
and hence in its complexification; the module structure on $\Delta$ therefore
restricts to a complex representation of $Spin(2n)$. The complex volume
element $\omega_{\C}=i^{n}e_{1}\cdots e_{2n}$ of the complexified Clifford algebra
(here $\{e_{1},\dots,e_{2n}\}$ is a positively oriented orthonormal basis of $\R^{2n}$,
and the product is independent of the basis chosen) satisfies $\omega_{\C}^{2}=1$ and
commutes with $Spin(2n)$. The eigenspaces $\Delta^{+}$ and $\Delta^{-}$
of $\omega_{\C}$ for the eigenvalues $+1$ and $-1$ are therefore subrepresentations
of $\Delta$, and $\Delta=\Delta^{+}\oplus\Delta^{-}$; these are the half-spin
representations of $Spin(2n)$ \cite[Ch.~IV, \textsection9]{LawsonMichelsohn}. Given a spin
structure on an oriented real vector bundle of rank $2n$, the associated bundles with
fibres $\Delta^{+}$ and $\Delta^{-}$ are its half-spinor bundles $S_{+}$ and $S_{-}$.

We use the following lemma in the proof of cases~\textup{(i)} and~\textup{(ii)} of
Theorem~\ref{thm:bundleacs}.
\begin{lemma}\label{lem:linesubbundle}
Let $M$ be a closed smooth manifold whose cohomology groups $H^{1}(M;\Z_{2})$,
$H^{2}(M;\Z_{2})$ and $H^{2}(M;\Z)$ vanish, let $n=2$ or $n=3$, and let $\xi$ be an
oriented real vector bundle of rank $2n$ over $M$. Then $\xi$ is spin and carries a
unique spin structure, and $\xi$ admits a complex structure if and only if its
positive half-spinor bundle $S_{+}$ admits a nowhere-vanishing section.
\end{lemma}

\begin{proof}
Since $\xi$ is oriented, $w_{1}(\xi)=0$, and $w_{2}(\xi)=0$
because $H^{2}(M;\Z_{2})=0$; thus $\xi$ is spin, and its spin structure is unique
because $H^{1}(M;\Z_{2})=0$ \cite[Ch.~II]{LawsonMichelsohn}. Fix a metric on $\xi$,
let $P$ be the bundle of oriented orthonormal frames of $\xi$ with respect to that
metric, which is a principal $SO(2n)$-bundle, and let $Q$ be the spin structure
of $\xi$, a principal $Spin(2n)$-bundle covering $P$.

Every complex structure on $\xi$ determines an orthogonal complex structure with
respect to a chosen metric $\langle\,\cdot\,,\cdot\,\rangle$, that is, an orthogonal endomorphism of $\xi$ squaring
to $-\mathrm{id}$ and inducing the given orientation. Indeed, a complex structure $J$ on $\xi$ is orthogonal for the metric $\langle\,\cdot\,,\cdot\,\rangle+\langle J(\cdot),J(\cdot)\rangle$. The space of metrics on $\xi$ is convex, so this new metric is joined to the chosen metric by a smooth path of metrics. As $M$ is compact, the homotopy invariance of fibre bundles implies that the bundles of orthogonal complex structures for the metrics at the two ends of this path are isomorphic. Therefore, $J$ yields an orthogonal complex structure for the chosen metric. Conversely an
orthogonal complex structure on $\xi$ makes it into a complex vector bundle of
rank $n$ whose underlying oriented real vector bundle is $\xi$. Thus $\xi$ admits a
complex structure if and only if it admits an orthogonal one.

Let $\mathscr{C}$ be the space of orthogonal complex structures on $\R^{2n}$ inducing
the standard orientation, on which $SO(2n)$ acts by conjugation. Write $P\times_{SO(2n)}\mathscr{C}$ for the twistor bundle of $\xi$, the bundle
over $M$ associated with $P$ and the action of $SO(2n)$ on $\mathscr{C}$. A frame $p\in P$ over $x$
carries $\mathscr{C}$ to the set of orthogonal complex structures on $\xi_{x}$
inducing the given orientation, compatibly with that action; the sections of
$P\times_{SO(2n)}\mathscr{C}$ are therefore precisely the orthogonal complex
structures on $\xi$.

Extend each $J\in\mathscr{C}$ complex linearly to $\C^{2n}=\R^{2n}\otimes\C$ and
let $V(J)=\{v\in\C^{2n}\colon Jv=-iv\}$, a subspace isotropic for the complex bilinear
extension of the inner product. Write $\mathscr{P}$ for the projectivisation of the space of positive pure spinors
in $\Delta^{+}$. By \cite[Ch.~IV, Proposition~9.7]{LawsonMichelsohn} the assignment carrying an element
of $\mathscr{P}$ to the maximal isotropic subspace $V$ annihilating it under Clifford
multiplication, and $V$ in turn to the unique $J\in\mathscr{C}$ with $V=V(J)$, is
an $SO(2n)$-equivariant diffeomorphism from $\mathscr{P}$ onto $\mathscr{C}$. 

As $2n\leq 6$, every non-zero element of $\Delta^{+}$ is pure \cite[Ch.~IV,
Remark~9.12]{LawsonMichelsohn}, so that $\mathscr{P}$ is the whole
projectivisation $\mathbb{P}(\Delta^{+})$ of $\Delta^{+}$. 
The covering homomorphism $Spin(2n)\to SO(2n)$ has kernel $\{\pm 1\}$, and $-1$ acts
on $\Delta^{+}$ as $-\mathrm{id}$, hence trivially on $\mathbb{P}(\Delta^{+})$; the
action of $Spin(2n)$ on $\mathbb{P}(\Delta^{+})$ thus factors through $SO(2n)$. Since $\mathscr{C}$ and $\mathbb{P}(\Delta^{+})$ are isomorphic $SO(2n)$-spaces
and $Q/\{\pm 1\}=P$, we obtain
\[
P\times_{SO(2n)}\mathscr{C}
\;\cong\;P\times_{SO(2n)}\mathbb{P}(\Delta^{+})
\;\cong\;Q\times_{Spin(2n)}\mathbb{P}(\Delta^{+})
\;\cong\;\mathbb{P}(S_{+}),
\]
where $\mathbb{P}(S_{+})$ denotes the fibrewise projectivisation of $S_{+}$. The
sections of $\mathbb{P}(S_{+})$ are precisely the complex line subbundles of $S_{+}$. Thus $\xi$ admits a complex structure if and only if $S_{+}$ has a complex
line subbundle.

Finally $H^{2}(M;\Z)=0$, so that every complex line bundle over $M$ is trivial. A
complex line subbundle of $S_{+}$ is thus spanned by a nowhere-vanishing section
of $S_{+}$, and conversely such a section spans one, completing the proof.
\end{proof}

We now prove Theorem~\ref{thm:bundleacs}. Cases~\textup{(i)} and~\textup{(ii)} are
deduced from Lemma~\ref{lem:linesubbundle}, case~\textup{(iii)} from
Theorem~\ref{thm:destabilize}, and case~\textup{(iv)} from the connectivity of the
maps of classifying spaces induced by the stabilisation
maps $U(n)\hookrightarrow U$ and $SO(2n)\hookrightarrow SO$.

\begin{proof}[\textbf{Proof of Theorem~\ref{thm:bundleacs}}]
Since $H^{1}(M_{\ell,m};\Z_{2})=H^{2}(M_{\ell,m};\Z_{2})=H^{2}(M_{\ell,m};\Z)=0$,
Lemma~\ref{lem:linesubbundle} applies to every oriented real vector bundle of
rank $4$ or $6$ over $M_{\ell,m}$; such a bundle is in particular spin, with a
unique spin structure, as is used in cases~\textup{(i)} and~\textup{(ii)}.

\smallskip
\noindent\textit{\textbf{Case \textup{(i)}: $2n=4$.}} Write $S_{+}$ and $S_{-}$
for the half-spinor bundles of $\xi$. Under the exceptional isomorphism
$Spin(4)\cong SU(2)\times SU(2)$~\cite[Ch.~I, Theorem~8.1]{LawsonMichelsohn}, the
half-spin representation $\Delta^{+}$ becomes the standard representation of one
$SU(2)$-factor and $\Delta^{-}$ that of the other, each factor acting trivially on
the representation of the other, so that $S_{+}$ and $S_{-}$ are $SU(2)$-bundles.

The vector representation of $Spin(4)$ on $\R^{4}$ complexifies to
$\Delta^{+}\otimes\Delta^{-}$, so that $\xi\otimes\C\cong S_{+}\otimes S_{-}$.
Write $\pm s$ and $\pm t$ for the Chern roots of $S_{+}$ and $S_{-}$ respectively, the
first Chern classes of both vanishing, so that $c_{2}(S_{+})=-s^{2}$
and $c_{2}(S_{-})=-t^{2}$. The Chern roots of $S_{+}\otimes S_{-}$ are $\pm(s+t)$
and $\pm(s-t)$, whence $c(\xi\otimes\C)=\bigl(1-(s+t)^{2}\bigr)\bigl(1-(s-t)^{2}\bigr)$,
and the relation $p_{r}(\xi)=(-1)^{r}c_{2r}(\xi\otimes\C)$ yields
\[
p_{1}(\xi)=-2\bigl(c_{2}(S_{+})+c_{2}(S_{-})\bigr),
\qquad
p_{2}(\xi)=\bigl(c_{2}(S_{+})-c_{2}(S_{-})\bigr)^{2}.
\]

One has $H^{*}\bigl(BSU(2);\Z\bigr)\cong\Z[c_{2}]$~\cite{hatcher}, where $c_{2}$
denotes the second Chern class of the universal $SU(2)$-bundle; the exceptional
isomorphism $Spin(4)\cong SU(2)\times SU(2)$ then yields
$H^{*}\bigl(BSpin(4);\Z\bigr)\cong\Z\bigl[c_{2}(\Delta^{+}),c_{2}(\Delta^{-})\bigr]$.
The Euler class of the universal oriented vector bundle of rank $4$
over $BSpin(4)$, associated with the vector representation, squares to its second
Pontryagin class \cite[Ch.~15]{characteristic}, which equals
$\bigl(c_{2}(\Delta^{+})-c_{2}(\Delta^{-})\bigr)^{2}$ by the computation above,
applied over $BSpin(4)$; the ring $\Z\bigl[c_{2}(\Delta^{+}),c_{2}(\Delta^{-})\bigr]$
being an integral domain, this Euler class
equals $\pm\bigl(c_{2}(\Delta^{+})-c_{2}(\Delta^{-})\bigr)$. Since $\xi$ is spin, it
is induced, together with $S_{+}$ and $S_{-}$, from the universal bundles
over $BSpin(4)$ by a classifying map; hence
\begin{equation}\label{eq:eulerhalfspin}
e(\xi)=\pm\bigl(c_{2}(S_{+})-c_{2}(S_{-})\bigr),
\end{equation}
the sign being the same for every oriented real vector bundle of rank $4$ over any of
the manifolds $M_{\ell,m}$.

The sign in the equation~\eqref{eq:eulerhalfspin} is determined by the bundle $\gamma_{\R}$
over $\HP^{2}=M_{1,0}$. Being the underlying real vector bundle of $\gamma$, it admits
a complex structure, so that by Lemma~\ref{lem:linesubbundle} its positive half-spinor
bundle $S_{+}$ has a nowhere-vanishing section; as $c_{2}(S_{+})$ is the Euler class of
the underlying oriented real vector bundle of $S_{+}$, it follows
that $c_{2}(S_{+})=0$. Together with
$c_{2}(S_{+})+c_{2}(S_{-})=-\tfrac{1}{2}\,p_{1}(\gamma_{\R})$,
equation \eqref{eq:gammaclasses} yields $c_{2}(S_{-})=-\alpha=e(\gamma_{\R})$;
as $\alpha\neq 0$, the sign in equation~\eqref{eq:eulerhalfspin} is negative. Therefore for the
bundle $\xi$, equation~\eqref{eq:eulerhalfspin} reads
$e(\xi)=c_{2}(S_{-})-c_{2}(S_{+})$; combining this with
$c_{2}(S_{+})+c_{2}(S_{-})=-\tfrac{1}{2}\,p_{1}(\xi)=-\sum_{i=1}^{k}a_{i}x_{i}$ yields
\begin{equation}\label{eq:rank4c2}
2c_{2}(S_{+})=-\sum_{i=1}^{k}a_{i}x_{i}-e(\xi).
\end{equation}

By Lemma~\ref{lem:linesubbundle} the bundle $\xi$ admits a complex structure if and
only if $S_{+}$ has a nowhere-vanishing section. If it does, then $c_{2}(S_{+})=0$,
this being the Euler class of the underlying oriented real vector bundle of $S_{+}$,
and equation~\eqref{eq:rank4c2} yields $e(\xi)=-\sum_{i=1}^{k}a_{i}x_{i}$. Conversely,
if $e(\xi)=-\sum_{i=1}^{k}a_{i}x_{i}$, then equation~\eqref{eq:rank4c2}
yields $2c_{2}(S_{+})=0$ and hence $c_{2}(S_{+})=0$, the
group $H^{4}(M_{\ell,m};\Z)$ being torsion-free; by Lemma~\ref{lem:su2} the
bundle $S_{+}$ is trivial and therefore has a nowhere-vanishing section. This proves
case~\textup{(i)}.

\smallskip
\noindent\textit{\textbf{Case \textup{(ii)}: $2n=6$.}} Under the exceptional
isomorphism $Spin(6)\cong SU(4)$~\cite[Ch.~I, Theorem~8.1]{LawsonMichelsohn}, the two
half-spin representations $\Delta^{\pm}$ of $Spin(6)$ become the standard
representation $\C^{4}$ of $SU(4)$ and its dual, while the vector representation
of $Spin(6)$ on $\R^{6}$ complexifies to $\Lambda^{2}\C^{4}$. The positive half-spinor
bundle $S_{+}$ of $\xi$ is therefore an $SU(4)$-bundle, in particular a complex vector
bundle over $M_{\ell,m}$ of rank $4$ with $c_{1}(S_{+})=0$, and
\[
\xi\otimes\C\;\cong\;\Lambda^{2}S_{+} .
\]
Moreover $c_{3}(S_{+})=0$, since $H^{6}(M_{\ell,m};\Z)=0$.

Let $y_{1},\dots,y_{4}$ be the Chern roots of $S_{+}$ and write $c_{j}=c_{j}(S_{+})$.
The Chern roots of $\Lambda^{2}S_{+}$ are the six sums $y_{i}+y_{j}$
with $1\le i<j\le 4$; as $y_{1}+\dots+y_{4}=0$, they occur in
pairs $\pm z_{1},\pm z_{2},\pm z_{3}$ with $z_{r}=y_{1}+y_{r+1}$, so that
$c(\Lambda^{2}S_{+})=\prod\limits_{r=1}^{3}(1-z_{r}^{2})$. Now
$z_{1}^{2},z_{2}^{2},z_{3}^{2}$ are the roots of
\begin{equation}\label{eq:cubic}
P(w)=w^{3}+2c_{2}w^{2}+(c_{2}^{2}-4c_{4})w-c_{3}^{2} .
\end{equation}
Indeed, put $w=(y_{1}+y_{2})^{2}$. Then $y_{1}y_{2}+y_{3}y_{4}=c_{2}+w$ and
$c_{3}=(y_{1}+y_{2})(y_{3}y_{4}-y_{1}y_{2})$, whence
$c_{3}^{2}=w\,(y_{3}y_{4}-y_{1}y_{2})^{2}$ and
\[
4c_{4}w
=4y_{1}y_{2}y_{3}y_{4}w=w\Bigl[(y_{1}y_{2}+y_{3}y_{4})^{2}-(y_{3}y_{4}-y_{1}y_{2})^{2}\Bigr]
=w(c_{2}+w)^{2}-c_{3}^{2},
\]
that is, $P(w)=0$; the same computation applies to $(y_{1}+y_{3})^{2}$
and $(y_{1}+y_{4})^{2}$.

Comparing coefficients in equation~\eqref{eq:cubic} yields $\sum_{r}z_{r}^{2}=-2c_{2}$
and $\sum_{r<s}z_{r}^{2}z_{s}^{2}=c_{2}^{2}-4c_{4}$, that is,
$c_{2}(\Lambda^{2}S_{+})=2c_{2}$ and $c_{4}(\Lambda^{2}S_{+})=c_{2}^{2}-4c_{4}$.
Since $p_{r}(\xi)=(-1)^{r}c_{2r}(\xi\otimes\C)$, we obtain
\begin{equation*}
p_{1}(\xi)=-2c_{2}(S_{+}),
\qquad
p_{2}(\xi)=c_{2}(S_{+})^{2}-4c_{4}(S_{+}).
\end{equation*}
Since $H^{4}(M_{\ell,m};\Z)$ is torsion-free, the first identity, combined
with $p_{1}(\xi)=2\sum_{i=1}^{k}a_{i}x_{i}$,
yields $c_{2}(S_{+})=-\sum_{i=1}^{k}a_{i}x_{i}$,
whence $c_{2}(S_{+})^{2}=\bigl(\sum_{i=1}^{k}\epsilon_{i}a_{i}^{2}\bigr)u$ by
Lemma~\ref{lem:cohomology}; the second identity, combined with $p_{2}(\xi)=b\,u$,
then yields
\begin{equation}\label{eq:rank6c4}
4c_{4}(S_{+})=\Bigl(\sum_{i=1}^{k}\epsilon_{i}a_{i}^{2}-b\Bigr)u .
\end{equation}

By Lemma~\ref{lem:linesubbundle} the bundle $\xi$ admits a complex structure if and
only if $S_{+}$ admits a nowhere-vanishing section. The obstructions to the existence
of such a section lie in $H^{q+1}\bigl(M_{\ell,m};\pi_{q}(\mathbb{S}^{7})\bigr)$
\cite[Ch.~12]{characteristic}, which vanishes for $q<7$, the
group $\pi_{q}(\mathbb{S}^{7})$ being trivial, and for $q>7$, the
manifold $M_{\ell,m}$ being $8$-dimensional. The only obstruction is therefore the
primary one, lying
in $H^{8}\bigl(M_{\ell,m};\pi_{7}(\mathbb{S}^{7})\bigr)\cong H^{8}(M_{\ell,m};\Z)$; it
is the Euler class of the underlying oriented real vector bundle of $S_{+}$, that
is, $c_{4}(S_{+})$. As $H^{8}(M_{\ell,m};\Z)\cong\Z$ is torsion-free,
equation~\eqref{eq:rank6c4} shows that $c_{4}(S_{+})$ vanishes if and only
if equation~\eqref{eq:rank6} holds. This proves case~\textup{(ii)}.

\smallskip
\noindent\textit{\textbf{Case \textup{({iii})}: $2n=8$.}} Here $\dim M_{\ell,m}=8=2n$
and $H^{8}(M_{\ell,m};\Z)\cong\Z$ has no $2$-torsion;
by Theorem~\ref{thm:destabilize} applied with $n=4$, the bundle $\xi$
admits a complex structure if and only if it admits a stable complex
structure $\omega$ with $c_{4}(\omega)=e(\xi)$. By Theorem~\ref{maintheorem} a stable complex structure
on $\xi$ exists if and only if the congruence~\eqref{maincong} holds, and by Lemma~\ref{lem:chern} every stable complex structure
$\omega$ on $\xi$ satisfies
$c_{4}(\omega)=\tfrac{1}{2}\bigl(b-\sum_{i=1}^{k}\epsilon_{i}a_{i}^{2}\bigr)u$.
Since $e(\xi)=e\,u$, the condition $c_{4}(\omega)=e(\xi)$ amounts
to equation~\eqref{eq:euler}. Hence $\xi$ admits a complex structure if and only if congruence~\eqref{maincong} and relation~\eqref{eq:euler} both hold. This proves case~\textup{({iii})}.

\smallskip
\noindent\textit{\textbf{Case \textup{({iv})}: $2n\geq 10$.}} A complex structure
on $\xi$ is in particular a stable complex structure, so that the congruence~\eqref{maincong} holds
by Theorem~\ref{maintheorem}. For the converse, assume the congruence~\eqref{maincong}
holds, so that by Theorem~\ref{maintheorem} the bundle $\xi$ admits a stable
complex structure $\omega$. The stabilisation map $U(n)\hookrightarrow U$ is
$2n$-connected~\cite{davis}, so that $BU(n)\to BU$ is
$(2n+1)$-connected. As $\dim M_{\ell,m}=8\leq 2n$, the induced map
\[
[M_{\ell,m},BU(n)]\longrightarrow[M_{\ell,m},BU]
\]
is bijective, so
there is a complex vector bundle $W$ of rank $n$ over $M_{\ell,m}$ whose stable class is
$\omega$. Then $W_{\R}$ and $\xi$ are stably isomorphic real vector bundles of rank $2n$, both oriented, $W_{\R}$ by the complex structure of  $W$. Choose an
isomorphism
$h:W_{\R}\oplus\varepsilon_{\mathbb{R}}^{s}\to\xi\oplus\varepsilon_{\mathbb{R}}^{s}$ for some
$s\geq 1$, the trivial summands carrying their standard orientations. As
$M_{\ell,m}$ is connected, $h$ either preserves or reverses orientation,
and in the second case $(\mathrm{id}_{\xi}\oplus\,r)\circ h$ preserves it,
where $r$ changes the sign of one coordinate of $\varepsilon_{\mathbb{R}}^{s}$. Hence
$W_{\R}$ and $\xi$ are stably isomorphic as oriented bundles. The stabilisation map $SO(2n)\hookrightarrow SO$ is
$(2n-1)$-connected~\cite{davis}, so that $BSO(2n)\to BSO$ is
$2n$-connected. Since $\dim M_{\ell,m}=8\leq 2n-1$, the induced map
\[
[M_{\ell,m},BSO(2n)]\longrightarrow[M_{\ell,m},BSO]
\]
is bijective, whence $W_{\R}\cong\xi$ as oriented bundles of rank $2n$. Thus $\xi$ admits a complex structure. This proves
case~\textup{({iv})}.
\end{proof}
\begin{remark}\label{rem:rank6realizability}
Let $\xi$ be an oriented real vector bundle of rank $6$ over $M_{\ell,m}$. A complex
structure on $\xi$ is in particular a stable one, so that the
congruence~\eqref{maincong} holds whenever the equation~\eqref{eq:rank6} does. As
conditions on arbitrary integer tuples, however, the equation~\eqref{eq:rank6} does
not imply the congruence~\eqref{maincong}. For $(\ell,m)=(1,0)$, so that
$M_{\ell,m}=\HP^{2}$, and for $a_{1}=2$ and $b=4$, the equation~\eqref{eq:rank6} holds
while the congruence~\eqref{maincong} fails. No contradiction arises, since not every
tuple $(a_{1},\dots,a_{k};b)$ is realised by a bundle of rank $6$.

Let $S_{+}$ be the positive half-spinor bundle of $\xi$, as in the proof of
Theorem~\ref{thm:bundleacs}\,\textup{(ii)}, and write
$[S_{+}]-4=\sum_{i=1}^{k}d_{i}\eta_{i}+n'\theta$ in the basis of
Lemma~\ref{lem:Ktheory}. Comparing Chern characters and using
$c_{1}(S_{+})=c_{3}(S_{+})=0$ yields $d_{i}=a_{i}$ and
$c_{4}(S_{+})=\bigl(\sum_{i=1}^{k}\epsilon_{i}\binom{a_{i}}{2}-6n'\bigr)u$; by
equation~\eqref{eq:rank6c4} this amounts to
\begin{equation}\label{eq:rank6realizable}
b=\sum_{i=1}^{k}\epsilon_{i}a_{i}(2-a_{i})+24n' .
\end{equation}

Conversely, every tuple $(a_{1},\dots,a_{k};b)$ with $b$ given by
the relation~\eqref{eq:rank6realizable} for some $n'\in\Z$ is realised by an oriented real vector
bundle of rank $6$ over $M_{\ell,m}$. Indeed, every oriented real vector bundle of
rank $6$ over $M_{\ell,m}$ is spin with a unique spin structure, so that the
assignment sending such a bundle to its positive half-spinor bundle is a bijection
from the set of isomorphism classes of oriented real vector bundles of rank $6$
over $M_{\ell,m}$ onto the set of isomorphism classes of $SU(4)$-bundles
over $M_{\ell,m}$; and since $H^{1}(M_{\ell,m};\Z)=H^{2}(M_{\ell,m};\Z)=0$, every
complex vector bundle of rank $4$ over $M_{\ell,m}$ admits a unique
$SU(4)$-structure, so that the latter set is in bijection with $[M_{\ell,m},BU(4)]$
and hence with $\widetilde{K}(M_{\ell,m})=[M_{\ell,m},BU]$, the map $BU(4)\to BU$
being $9$-connected and $M_{\ell,m}$ of dimension $8$. Given $a_{1},\dots,a_{k}$
and $n'$, the class $\sum_{i=1}^{k}a_{i}\eta_{i}+n'\theta$ therefore comes from such a
bundle, whose corresponding tuple is
$\bigl(a_{1},\dots,a_{k};\sum_{i=1}^{k}\epsilon_{i}a_{i}(2-a_{i})+24n'\bigr)$.

Taking $a_{1}=\dots=a_{k}=0$ and $n'=1$ in the equation~\eqref{eq:rank6realizable}, for instance,
yields an oriented real vector bundle of rank $6$ over $M_{\ell,m}$ with $b=24$. The
congruence \eqref{maincong} holds for it, since $24\equiv 0\pmod 4$, so that it admits
a stable complex structure; the equation \eqref{eq:rank6} fails, however, since
$\sum_{i=1}^{k}\epsilon_{i}a_{i}^{2}=0\neq 24$, so that it admits no complex structure.
The two criteria are therefore distinct.

Suppose now that the equation~\eqref{eq:rank6} holds. Comparing it with the equation~\eqref{eq:rank6realizable} yields 
\[
\sum_{i=1}^{k}\epsilon_{i}a_{i}(a_{i}-1)=12n',\]
and therefore
\[
b-\sum_{i=1}^{k}\epsilon_{i}a_{i}(2a_{i}-1)
=-\sum_{i=1}^{k}\epsilon_{i}a_{i}(a_{i}-1)=-12n'\;\equiv\;0 \pmod 4 ,
\]
so that the congruence~\eqref{maincong} holds.

The same holds in rank $4$, with condition~\eqref{eq:rank4nec} in place of condition~\eqref{eq:rank6}.
Let $\xi$ be an oriented real vector bundle of rank $4$ over $M_{\ell,m}$
satisfying equation~\eqref{eq:rank4nec}. The second Pontryagin class of such a bundle being
the square of its Euler class, $\xi$ satisfies equation \eqref{eq:rank6}. Pontryagin classes
being stable, the bundle $\xi\oplus\varepsilon_{\R}^{2}$ has rank $6$ and the same associated
integers $a_{1},\dots,a_{k}$ and $b$ as $\xi$, and therefore
satisfies equation~\eqref{eq:rank6} as well; hence by the preceding paragraph it
satisfies the congruence~\eqref{maincong}, which is a condition on $a_{1},\dots,a_{k}$ and $b$
alone, and so holds for $\xi$ as well. As a condition on arbitrary integer tuples, however, equation~\eqref{eq:rank4nec} no more
implies the congruence~\eqref{maincong} than equation~\eqref{eq:rank6} does. The tuple considered at the beginning of this remark satisfies equation~\eqref{eq:rank4nec}
when the Euler class is taken to be $-2x_{1}$, while congruence~\eqref{maincong} fails for it;
since congruence~\eqref{maincong} holds for every bundle of rank $4$
satisfying equation~\eqref{eq:rank4nec}, no such bundle over $\HP^{2}$ realises the tuple.
\end{remark}
\section{Almost complex structures on \texorpdfstring{$M_{\ell,m}$}{M\_\{l,m\}}}\label{sec:acs}

We now use Theorem~\ref{thm:bundleacs} to determine which of the manifolds
$M_{\ell,m}$ admit an almost complex structure, applying it to the tangent
bundle. Throughout this section an almost complex structure on $M_{\ell,m}$ is
understood to induce the given orientation. As $TM_{\ell,m}$ has rank $8$, the
relevant case is~\textup{(iii)}, and the Euler class of the tangent bundle is
$e(TM_{\ell,m})=\chi(M_{\ell,m})\,u$.

By Lemma~\ref{lem:cohomology} the cohomology $H^{*}(M_{\ell,m};\Z)$ is
torsion-free and concentrated in degrees $0$, $4$ and $8$, of ranks $1$, $k$ and
$1$ respectively; all odd Betti numbers therefore vanish, and
\begin{equation}\label{eq:chi}
\chi(M_{\ell,m})=1+k+1=\ell+m+2 .
\end{equation}
Corollary~\ref{thm:acsclassification} now follows by substituting the invariants of
$TM_{\ell,m}$ into Theorem~\ref{thm:bundleacs}\,({iii}).
\begin{proof}[\textbf{Proof of Corollary~\ref{thm:acsclassification}}]
By Theorem~\ref{thm:bundleacs}\,({iii}), $M_{\ell,m}$ admits an almost
complex structure if and only if $\xi=TM_{\ell,m}$ satisfies both
\eqref{maincong} and equation~\eqref{eq:euler}. By Theorem~\ref{prop:p1ConnectedSum} the
tuple attached to $\xi$ in equation~\eqref{eq:invariants} is
$(a_{1},\dots,a_{k};b)=(1,\dots,1;7(\ell-m))$, while $e=\ell+m+2$ by
equation~\eqref{eq:chi}.

Since $\sum_{i=1}^{k}\epsilon_{i}a_{i}^{2}=\sum_{i=1}^{k}\epsilon_{i}=\ell-m$,
condition \eqref{eq:euler} reads
\[
7(\ell-m)-(\ell-m)=2(\ell+m+2),
\]
that is, $\ell=2m+1$. The congruence \eqref{maincong}, on the other hand, holds
for $TM_{\ell,m}$ precisely when $\ell-m$ is even, by
Corollary~\ref{cor:stablyAC}.

The two conditions hold simultaneously if and only if $\ell=2m+1$ and $m+1$ is
even, that is, if and only if $\ell=2m+1$ and $m$ is odd. This completes the
proof.
\end{proof}
\begin{remark}\label{rem:Yang2012acs}
As in Remark~\ref{rem:Yang2012stable}, the criterion of
Yang~\cite[Theorem~2, case~1]{Yang2012} applies to $M_{\ell,m}$ with $n=4$. According to it,
$M_{\ell,m}$ admits an almost complex structure if and only if it admits a stable
almost complex structure and
\[
4\bigl\langle p_{2}(M_{\ell,m}),[M_{\ell,m}]\bigr\rangle
-\bigl\langle p_{1}(M_{\ell,m})^{2},[M_{\ell,m}]\bigr\rangle=8(\ell+m+2).
\]
By
Theorem~\ref{prop:p1ConnectedSum} and Lemma~\ref{lem:cohomology} the left-hand side
equals $24(\ell-m)$, so that the condition above reads $\ell=2m+1$. Together with Corollary~\ref{cor:stablyAC}, this recovers
Corollary~\ref{thm:acsclassification}.
\end{remark}
\begin{remark}\label{rem:SatoSuzuki}
Sato and Suzuki \cite[Theorem~B]{SatoSuzuki} assert that the connected sums
$\ell\,\HP^{n}\mathbin{\#}m\,\overline{\HP^{n}}$ admit no almost complex
structure for $n=1,2,4,5,\dots,10$. For $n=2$ this family is precisely
$M_{\ell,m}$, and the assertion is not correct, since by
Corollary~\ref{thm:acsclassification} with $j=0$ the manifold
$3\,\HP^{2}\mathbin{\#}\overline{\HP^{2}}$ does admit an almost complex
structure.

As a necessary condition for the existence of an almost complex structure on
these connected sums, they derive
\begin{equation}\label{eq:SS}
\ell\bigl\{s_{n}-(n-1)\bigr\}=m\bigl\{s_{n}+(n-1)\bigr\}+2 ,
\end{equation}
where $s_{n}$ is the coefficient of $x^{n}$ in $(1-x)^{n+1}(1-4x)^{-1/2}$; it is
stated there in terms of $s_{n}/(n+1)$. For $n=2$ one computes $s_{2}=3$, so that the relation~\eqref{eq:SS} becomes $2\ell=4m+2$, that is, $\ell=2m+1$, which is
precisely the first of the two conditions of
Corollary~\ref{thm:acsclassification}. The nonexistence assertion therefore does
not follow from the relation~\eqref{eq:SS}. For $n=2$ it is asserted in \cite{SatoSuzuki} by
an appeal to Heaps \cite{Heaps}, no argument being given. The criterion of
\cite[Theorem~1]{Heaps} for closed oriented $8$-manifolds returns, when applied
to $M_{\ell,m}$, the two conditions of
Corollary~\ref{thm:acsclassification}, as we now verify.

As $H^{2}(M_{\ell,m};\Z)$ and $H^{6}(M_{\ell,m};\Z)$ vanish, the two cohomology
classes occurring in~\cite[Theorem~1]{Heaps} are zero, and its conditions reduce to
$w_{8}(M_{\ell,m})=0$ together with
$8\chi(M_{\ell,m})=\bigl\langle 4p_{2}(M_{\ell,m})-p_{1}(M_{\ell,m})^{2},[M_{\ell,m}]\bigr\rangle$.
By Theorem~\ref{prop:p1ConnectedSum} and equation~\eqref{eq:chi} the second reads
$8(\ell+m+2)=24(\ell-m)$, that is, $\ell=2m+1$, the condition already obtained
from the relation~\eqref{eq:SS}. Since $w_{8}(M_{\ell,m})$ is the mod~$2$ reduction of
$e(TM_{\ell,m})=\chi(M_{\ell,m})\,u$, the first says that
$\chi(M_{\ell,m})=\ell+m+2$ is even, and hence, once $\ell=2m+1$ is assumed, that
$m$ is odd. By Corollary~\ref{cor:stablyAC} it holds exactly when $TM_{\ell,m}$
is stably almost complex, and it is precisely this condition that the relation~\eqref{eq:SS}
does not detect.
\end{remark}
\begin{remark}\label{rem:orientation}
Corollary~\ref{thm:acsclassification} concerns almost complex structures inducing the given orientation. Reversing the orientation replaces $M_{\ell,m}$ by $M_{m,\ell}$, so the underlying smooth manifold admits an almost complex structure if and only if either $\ell=2m+1$ with $m$ odd, or $m=2\ell+1$ with $\ell$ odd. Since these two conditions are mutually exclusive, no such manifold $M_{\ell,m}$ is almost complex with respect to both of its orientations.
\end{remark}
\begin{acknow}
The author is sincerely grateful to Ramesh Kasilingam for his
valuable suggestions, comments, and probing questions, which have
improved the quality of this article.
The author also
gratefully acknowledges financial support in the form of Prime Minister’s Research Fellowship, Government of India (PMRF/2502403).
\end{acknow}
\bibliographystyle{alpha} 
\bibliography{reference}
\end{document}